\documentclass{article}

\usepackage[utf8]{inputenc}
\usepackage{float}
\usepackage{url}
\usepackage[margin =1in]{geometry}
\usepackage{amssymb}
\usepackage{amsthm}
\usepackage{amsmath}
\usepackage{enumerate}
\usepackage{subcaption}
\usepackage{cite}
\usepackage{giraffemath}
\usepackage{algorithm}
\usepackage[ruled,vlined,linesnumbered,algo2e]{algorithm2e}
\usepackage{breqn}
\usepackage{etoolbox}%

\usepackage{xcolor}
\usepackage{hyperref}
\usepackage{cleveref}
\newtheorem{theorem}{Theorem}
\newtheorem{definition}[thm]{Definition}
\newtheorem{proposition}[thm]{Proposition}

\newcommand{\ml}[1]{{{\color{black}{{#1}}}}}
\hypersetup{
    colorlinks=true,
    linkcolor=blue,
    filecolor=magenta,      
    urlcolor=cyan,
    pdftitle={Overleaf Example},
    pdfpagemode=FullScreen,
    }
\usepackage{bbm}

\def\E{\mathbb{E}}
\def\P{\mathbb{P}}
\def\one{\mathbbm{1}}
 
\newtheorem{inequality}{Inequality}

\crefname{assumption}{assumption}{assumptions}

\begin{document}
\title{A Stochastic Quasi-Newton Method in the Absence of Common Random Numbers}

\author{Matt Menickelly\thanks{Argonne National Laboratory, Lemont, IL, USA. {mmenickelly@anl.gov}.}
\and 
Stefan M.\ Wild\thanks{Applied Math \& Computational Research Div., 
Lawrence Berkeley National Laboratory, Berkeley, CA,
USA {wild@lbl.gov}.}
\and Miaolan Xie\thanks{School of Industrial Engineering, Purdue University, West Lafayette, IN, USA. 
{xie537@purdue.edu}.}
}



\maketitle

\begin{abstract}
We present a quasi-Newton method for unconstrained stochastic optimization. Most existing literature on this topic assumes a setting of stochastic optimization in which a finite sum of component functions is a reasonable approximation of an expectation, and hence one can design a quasi-Newton method to exploit common random numbers. In contrast, and motivated by problems in variational quantum algorithms, we assume that function values and gradients are available only through inexact probabilistic zeroth- and first-order oracles and no common random numbers can be exploited. Our algorithmic framework--based on prior work on the  SASS algorithm--is general and does not assume common random numbers. We derive a high-probability tail bound on the iteration complexity of the algorithm for nonconvex and strongly convex functions. We present numerical results demonstrating the empirical benefits of augmenting SASS with our quasi-Newton updating scheme, both on synthetic problems and on real problems in quantum chemistry. 
\end{abstract}

\maketitle

\section{Introduction}

Quasi-Newton (QN) methods have a history in deterministic optimization that now exceeds 50 years, including the celebrated papers by Broyden \cite{Broyden1970}, Fletcher \cite{Fletcher1970}, Goldfarb \cite{Goldfarb1970}, and Shanno \cite{Shanno1970}. 
Classically, QN methods are designed for the unconstrained minimization of a strongly convex and twice continuously differentiable objective function $\phi: \R^n \to \R$.
The usual direction $d_k$ in the $k$th iteration of a standard Newton's method in such a setting would be defined by
$d_k=B_k^{-1}\nabla \phi(x_k)$,
with  $B_k = \nabla^2 \phi(x_k)$.
QN methods, however, assume that the Hessian inverse is computationally expensive or otherwise unavailable, and they seek to replace $H_k\triangleq B_k^{-1}$ with a matrix updated at every iteration with locally relevant information concerning the curvature of $\phi$.
 In particular, before the iteration counter augments from $k$ to $k+1$, 
 QN methods essentially select a matrix $B_{k+1}$ that satisfies the secant condition
 \begin{equation}\label{eq:secant_eq}
 B_{k+1}(x_{k+1}-x_k) = \nabla \phi(x_{k+1}) - \nabla \phi(x_k)
 \end{equation}
 and, given a precise metric, is closest to $B_k$. 
 When step sizes are properly controlled, and under appropriate assumptions on $\phi$ and the starting point $x_0$, employing such updates results in superlinear convergence of a QN method; see Chapter 6 in \cite{NW} for a detailed introduction to this topic.
 
 To aid our discussion, we  introduce two pieces of notation,
 \begin{equation}
 \label{eq:sydef}
 s_k \triangleq x_{k+1}-x_k \quad \text{ and } \quad y_k \triangleq g(x_{k+1}) - g(x_k),
 \end{equation}
where $g(x_k)$ and $g(x_{k+1})$ are some approximations of $\nabla \phi(x_{k})$ and $\nabla \phi(x_{k+1})$, respectively.  Observe that when we replace $g(x_{k+1}), g(x_k)$ in \cref{eq:sydef} with $\nabla \phi(x_{k+1}), \nabla \phi(x_k)$, \cref{eq:secant_eq} simplifies to $B_{k+1}s_k = y_k$. 

Empirical risk minimization is critical in machine learning.
In such a setting, a typical objective function has the form 
\begin{equation}
    \label{eq:erm}
    \phi(x) = \displaystyle\frac{1}{N} \displaystyle\sum_{i=1}^N F(x;\xi_i),
\end{equation}
where the \emph{loss function} $F$ is a function of both model parameters $x$ and a data point drawn from a (large) finite set $\{\xi_1,\xi_2,\dots,\xi_N\}$; $
\phi$ is sometimes assumed to be (strongly) convex in $x$. 
While one could simply apply a QN method directly to the minimization of \cref{eq:erm}, this is often undesirable in practice because computing the full gradient $\nabla \phi(x_k)$ at every iteration involves a gradient evaluation of $N$ distinct loss functions. 
Thus, in recent decades, researchers have applied 
\emph{minibatching} from the stochastic gradient literature to QN methods. 
Such methods define the gradient estimator $g(x)$ in \cref{eq:sydef} as
$$g_{\mathcal{B}}(x) = \displaystyle\frac{1}{|\mathcal{B}|} \sum_{i\in\mathcal{B}} \nabla F(x; \xi_i),$$
where $\mathcal{B}\subseteq\{1,2,\dots,N\}$. 
By employing secant updates with a \emph{common} batch of indices on each iteration, that is, by replacing
$$y_k = g_{\mathcal{B}_{k}}(x^{k+1}) - g_{\mathcal{B}_{k}}(x_k),$$
theoretically and empirically useful stochastic quasi-Newton methods can be derived. 
For an excellent survey on this topic, see \cite{Mokhtari2020}. 
We draw particular attention to 
online L-BFGS (oLBFGS) \cite{Bordes2009},
regularized stochastic BFGS (RES) \cite{Mokhtari2014},
stochastic L-BFGS (SLBFGS) \cite{moritz2016linearly},  
and incremental quasi-Newton methods (IQNs) \cite{Mokhtari2018};
we note that in \cite{Mokhtari2018} the authors demonstrate that IQN exhibits a superlinear convergence rate under appropriate assumptions, thereby recovering deterministic guarantees up to constants. 

We also note that the stochastic quasi-Newton methods so far described are not precisely for stochastic optimization; instead, they are \emph{randomized} methods for solving \emph{deterministic} optimization problems. 
These stochastic quasi-Newton methods fundamentally employ common random numbers (CRNs) by allowing for control on batches $\mathcal{B}_k$ when updating the matrix $B_{k+1}$. 
Broadly speaking, optimization methods employing CRNs in stochastic optimization enjoy better worst-case complexity than those that do not employ CRNs. 
As an example, in zeroth-order optimization, this is well documented in the study of one-point (i.e., CRN-free) bandit methods 
(see, e.g., \cite{Flaxman2005,Bach2016}) versus two-point (i.e., employing CRNs) bandit methods (see, e.g., \cite{Duchi2015}). 
Thus, although we do not prove any information-theoretic lower bounds in this paper, it is not reasonable to expect non-asymptotic superlinear convergence of a QN method in a CRN-free setting with weak assumptions on the input (introduced in \Cref{sec:assumptions}) as presented here; 
the best results we will demonstrate guarantee linear convergence with overwhelmingly high probability. 

On the theoretical side, this paper is related to a recent line of work \cite{gratton2018complexity, jin2021high, cao2022first, scheinberg2022stochastic, berahas2023sequential} that considers possibly biased stochastic oracles as input for various optimization problems and methods and provides iteration complexity bounds with overwhelmingly high probability. We contribute to this body of work by providing the first high-probability analysis for the quasi-Newton method in the setting where function gradients and values are accessible only through some biased probabilistic oracles and no CRNs can be exploited.

\subsection{Motivation -- Variational Quantum Algorithms}
\label{sec:motivation}
In this paper, we are interested in the minimization of a function where function information is accessible only via stochastic oracles, with no ability to employ CRNs. 
We are particularly motivated by problems in variational quantum algorithms (VQAs) \cite{Cerezo2021}. 
Building on past work \cite{menickelly2022latency}, we focus on one type of VQA, the variational quantum eigensolver (VQE), relevant in applications such as quantum chemistry, although the work presented in this paper is relevant for virtually any VQA.
VQE seeks the ground state energy of a closed quantum system. 
It does so by parameterizing a wavefunction via a (finite) set of variables $x$ so that a sufficiently flexible and expressive ansatz $\psi(x)$ may stand in for the wavefunction.
Then, given the system Hamiltonian $H$, the resulting optimization problem in standard optimization notation\footnote{Quantum physicists would write this objective function in bra-ket notation, that is, $\phi(x) = \langle \psi(x)| H | \psi(x) \rangle$.} is 
\begin{equation}\label{eq:vqe}
\displaystyle\min_x \phi(x) = \psi(x)^\top H \psi(x).
\end{equation}
By inspection, \cref{eq:vqe} is simply the minimization of an (unnormalized) Rayleigh quotient and represents finding the least eigenvalue of $H$, which corresponds to the ground state energy of the system. 

Consistent with the paradigm of quantum computing, evaluations of $\phi$ on a quantum computer cannot be made directly.
Rather, even in an idealized, noise-free quantum computer, evaluations of $\phi$ are available only as independent stochastic \emph{unbiased} estimates $f(x,\xi)\approx\phi(x)$, where the random variable $\Xi$ with realizations $\xi$ has a distribution determined by the moduli of complex coefficients that define a superposition of qubit states. 
Because quantum gates are defined by unitary matrices via the evolution postulate of quantum mechanics, 
one can show (see, e.g., \cite{Schuld2019}) that for many classes of ansatz $\psi(x)$, 
partial derivatives of $\phi$ can be computed via
\begin{equation}
    \label{eq:magic_dd}
    \partial_i \phi(x) = \displaystyle\frac{\phi(x + \frac{\pi}{2}e_i) - \phi(x - \frac{\pi}{2}e_i)}{2},
\end{equation}
where $e_i$ is the $i$th standard basis vector.
In other words, partial derivatives of many VQE objectives are \emph{exactly equal to} a multiple of a specific central finite-difference gradient estimate. 
Moreover, one can obtain an \emph{unbiased} gradient estimate of $\nabla \phi(x)$ via $2n$ \emph{unbiased} function estimations of $\phi$.
We stress unbiasedness here because it is well known that general finite-difference stochastic gradient estimates yield biased estimators of the true gradient (see \cite{berahas2022theoretical} for a recent treatment of this subject). 

On a quantum computer, realizations $\xi$ are purely exogenous and outside of the optimizer's control.
That is, in the language of stochastic optimization, quantum computers are CRN-free, and hence the quasi-Newton methods proposed to date are inapplicable in this setting. 
This situation motivates our subsequent development of a CRN-free quasi-Newton method.

We also highlight that the quantum computing community has gravitated to the deployment of standard optimization methods for the solution of VQA-related problems such as \cref{eq:vqe}. 
For instance, the software package \texttt{Qiskit} \cite{Qiskit} implements several gradient-based methods, including standard gradient descent methods, conjugate gradient methods, and the code \texttt{L-BFGS-B} \cite{zhu1997algorithm}. 
However, since these methods assume \emph{deterministic} and \emph{exact} access to $\nabla\phi(x)$, they are inappropriate for performing stochastic optimization. We highlight this in \Cref{fig:lbfgs_failure}, 
where we consider two particular instances of VQE problems and increase along the $x$-axis the number of realizations of samples used to form the estimates of the gradient (by estimating the function values in \cref{eq:magic_dd} increasingly more accurately via a larger sample average) that is supplied to L-BFGS-B. 

\begin{figure}[ht]\centering
	\begin{subfigure}{.49\textwidth}
		\centering
		\includegraphics[trim=2 -20 20 0, clip, width=\linewidth]{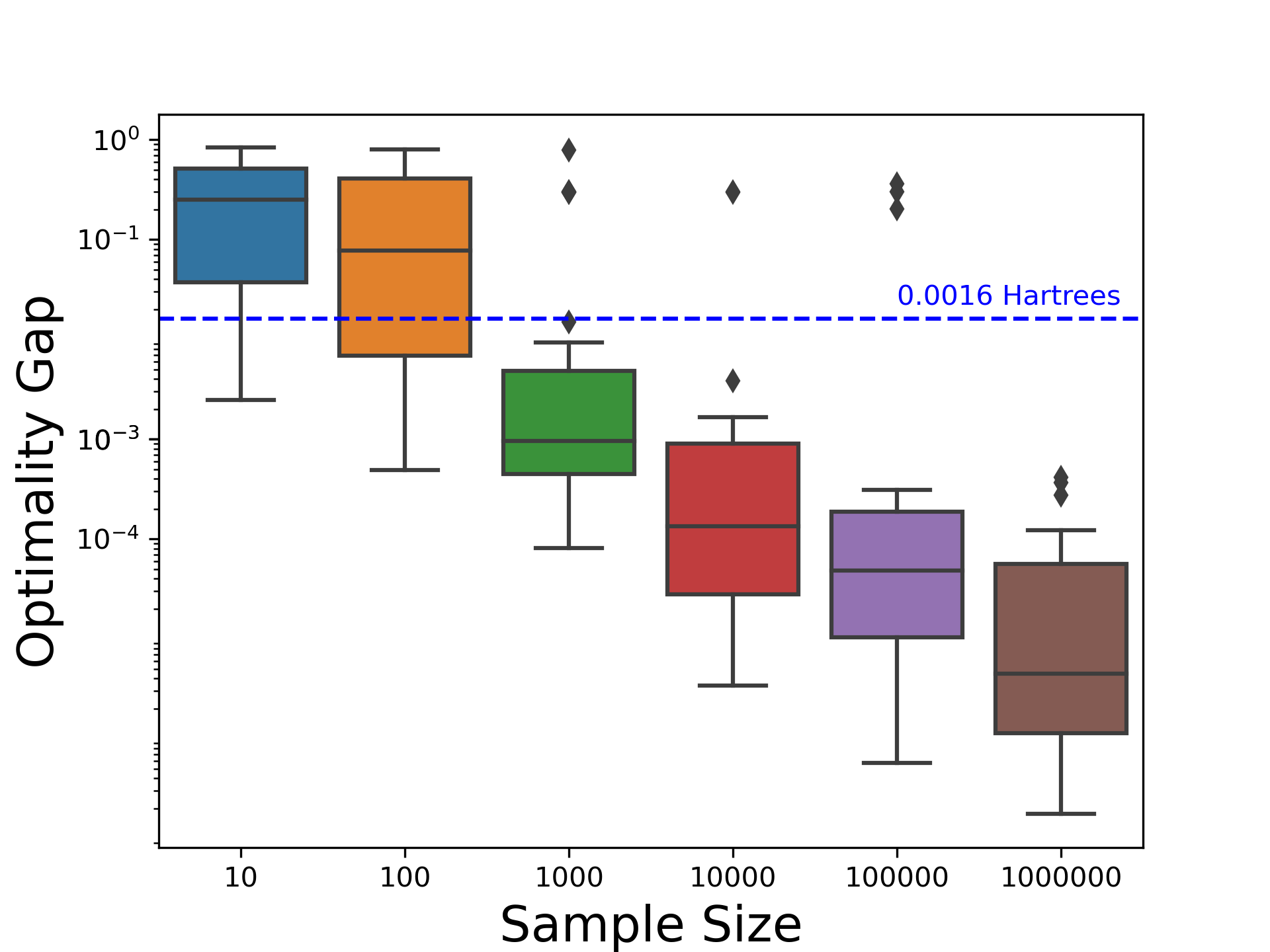}
	\end{subfigure}
	\hfill
	\begin{subfigure}{.49\textwidth}
		\centering
		\includegraphics[trim=2 -20 20 0, clip, width=\linewidth]{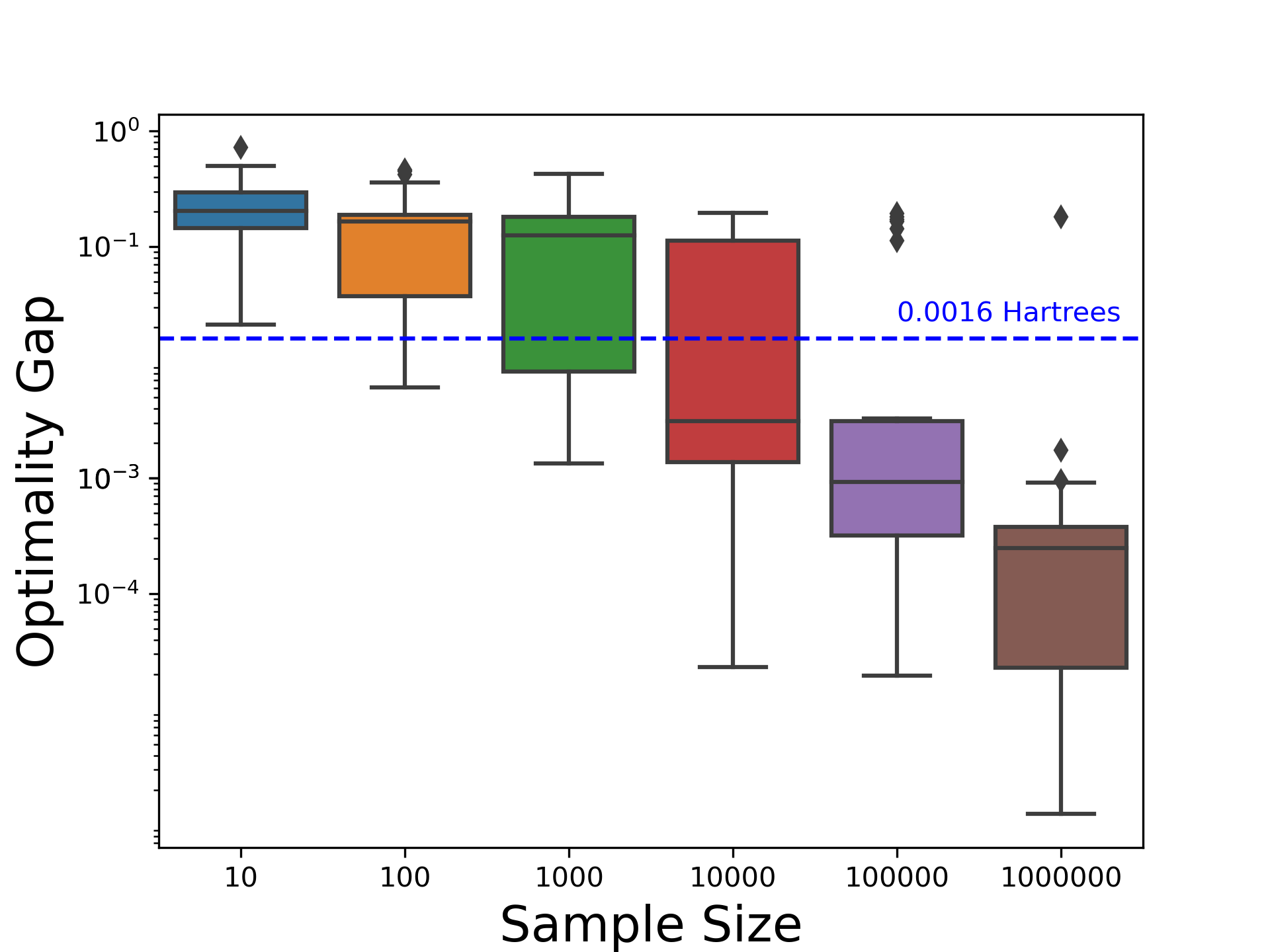}
	\end{subfigure}
	\caption{
Box plots summarizing the optimality gaps of the final solutions returned by 30 runs of L-BFGS-B for different sample budgets with the same initial point. 
The left plot corresponds to a Hamiltonian for an H$_2$ molecule, while the right plot corresponds to a Hamiltonian for a LiH molecule.
}
	\label{fig:lbfgs_failure}
\end{figure}

We remark that as of March 2024, the default number of samples $f(x,\xi)$ that is requested in the cloud-based computing service IBM-Q in a single batch is 4,000 \cite{IBMQ-FAQ};
 thus, the sample sizes of 1,000 and 10,000 in \Cref{fig:lbfgs_failure} are of particular interest.
In general, and as one should expect, the more samples L-BFGS-B uses per evaluation, the better the overall performance is. 
However,  the output of a VQE is essentially meaningless unless it achieves  \emph{chemical accuracy} \cite{helgaker2014molecular}, which translates to an optimality gap of $0.0016$ Hartrees. 
We note in \Cref{fig:lbfgs_failure} that while for the smaller molecule H$_2$ a majority of runs achieve a solution within chemical accuracy, this is certainly not the case for the larger LiH molecule. 
Thus, we stress that we require an L-BFGS method that not only is CRN-free but also is supplied with \emph{convergence guarantees} in terms of the quality of the solution returned.

\subsection{Our Contributions}
In this paper, we propose a CRN-free quasi-Newton method for unconstrained stochastic optimization. The algorithm is based on SASS (Stochastic Adaptive Step Search \cite{jin2021high}), and we call it Q-SASS (Quasi-Newton SASS). 
Q-SASS is a powerful stochastic quasi-Newton optimization method that is demonstrably useful for solving VQA problems on quantum computers, and other optimization problems where the underlying stochasticity cannot be controlled. 
These problems cannot be addressed using traditional stochastic quasi-Newton methods that are CRN-based, making Q-SASS a valuable tool for tackling them. 
We prove that for nonconvex functions, Q-SASS obtains an $\varepsilon$-stationary point in $O(\frac{1}{\varepsilon^2})$ iterations with overwhelmingly high probability.
We additionally prove that for strongly convex functions, Q-SASS converges linearly to an $\varepsilon$-optimal point with an overwhelmingly high probability. 
The advantage of the algorithm is confirmed by empirical experiments.

In the next section we introduce our problem setting and assumptions. In \Cref{sec:algo} the CRN-free Q-SASS algorithm is introduced. \Cref{sec:analysis} introduces the general analysis framework and analyzes the iteration complexity of this stochastic quasi-Newton method for both nonconvex functions and strongly convex functions. Numerical results on both synthetic problems and real problems in quantum chemistry are presented in \Cref{experiment}. We conclude with a summary in \Cref{sec: conclusion}.

\section{Problem Setting and Oracles}\label{sec:assumptions}
We consider the unconstrained minimization of a function $\phi:\R^n\to\R$. We begin by making a standard assumption on the function $\phi$. 
\begin{assumption}\label{ass:Lipschitz}
	The objective function $\phi:\R^n\to\R$ is bounded from below by some constant $\phi^*$, and the gradient function $\nabla \phi$ is $L$-Lipschitz continuous. 
\end{assumption}  

While we will always be assuming \Cref{ass:Lipschitz} in what follows, we will sometimes make the following additional assumption in order to provide stronger results in a strongly convex setting. 
\begin{assumption}\label{ass:strongly_convex}
	The objective function $\phi$ is $\beta$-strongly convex.
	That is,
	$$\phi(x) \geq \phi(y) + \nabla\phi(y)^T(x-y) + \frac{\beta}{2}\norm{x-y}^2, \quad \text{for all $x, y \in \R^n$}.$$
\end{assumption}
We remark that we could, in fact, relax \Cref{ass:strongly_convex} to satisfy the Polyak--{\L}ojasiewicz inequality \cite{polyak1963gradient}.


We now discuss our optimization framework. At a high level, the framework is comprised of two entities: The algorithm, and the oracles. Any information related to the function is provided to the algorithm by querying the oracles. 
Throughout, we assume that our algorithm has access to a \emph{zeroth-order oracle} and a \emph{first-order oracle}. 
When queried at a point $x$, the zeroth-order oracle returns some estimate of $\phi(x)$ and the first-order oracle returns some estimate of $\nabla\phi(x)$. 

To prove meaningful convergence results, one must naturally make assumptions concerning the oracles, but one wishes for these assumptions to be as broad as allowable to encapsulate defects in estimates such as bias, noise, or adversarial corruption. 
For instance, expected risk minimization problems in machine learning may encounter outliers and distribution shifts; federated learning can experience system errors or failures, and even be subjected to adversarial attacks; derivative-free optimization problems may face measurement errors; and in both simulation optimization and VQA problems, noise may hamper progress due to insufficient sampling. In all of these settings, one only has access to possibly unreliable estimates of the function values and gradients. 

Oracles can be constructed in a variety of ways. For example, in empirical risk minimization, the zeroth- and first-order oracle are typically derived from averaging the loss function values and gradients over a minibatch drawn from a fixed library of data samples.
In VQA problems as described in \Cref{sec:motivation}, the zeroth-order oracle is obtained by averaging independent stochastic estimates of $\phi(x)$, and in many settings, a first-order oracle can be obtained from the zeroth-order oracle via \eqref{eq:magic_dd}. The precise, but sufficiently general, requirements on the oracles needed for our theoretical results are encoded in the following two assumptions.

\begin{assumption}\label{ass:zeroth_order_oracle}
	{\bf Probabilistic zeroth-order oracle.}
	Given  a point $x\in\R^n$,  the oracle computes  $f(x,\Xi)$, with realization $f(x,\xi)$, a (random) estimate of the function value $\phi(x)$. 
	The second argument $\Xi$, is a random variable (whose distribution may depend on $x$), with probability space $(\Omega, \mathcal{F}_{\Omega}, P)$.  
The absolute value of the estimation error $e(x,\Xi)$, with realization $e(x,\xi)=|f(x,\xi)-\phi(x)|$, satisfies, for all $x\in\R^n$, 
	\begin{equation}\label{eq:zero_order1}
		{\mathbb E_\Xi}\left [ e(x,\Xi)\right ]\leq \varepsilon_f \ \text{and }
	\end{equation}
	\begin{equation}\label{eq:zero_order2}
		{\mathbb E_\Xi}\left [\exp\{\lambda (e(x,\Xi)-\E_{\zeta}[e(x,\zeta)])\}\right ]\leq \exp\left(\frac{\lambda^2\nu^2}{2}\right), \quad \forall \lambda\in \left[0,\frac{1}{b}\right]
	\end{equation}
	for some non-negative parameters $(\varepsilon_f,\nu,b)$.
	\end{assumption}
	
 In the requirements for the zeroth-order oracle, \eqref{eq:zero_order1} simply states that the first moment of $e(x,\Xi)$ is bounded by $\varepsilon_f$, and 
	\eqref{eq:zero_order2} states that the noise in the function value estimates should be light-tailed. Note that \eqref{eq:zero_order2} is a weaker assumption than $e(x,\Xi)$ being subexponential, since \eqref{eq:zero_order2}  requires only  the right tail of $e(x,\Xi)$ decay at least exponentially.
 In particular, \eqref{eq:zero_order1} and \eqref{eq:zero_order2} taken together imply that we never require arbitrarily accurate zeroth-order oracles; as a result, one should not expect that a convergence result of any method employing such an oracle should offer convergence to a stationary point.   

We remark that in the definition above, $\Xi$ captures all sources of randomness in the oracle, and it may depend on $x$ and other parameters. 
    To continue with our two examples, in empirical risk minimization, $\Xi$ represents the random minibatch of samples used to estimate the function value.
    In a VQA setting, $\Xi$ is derived from measurements of $\phi(x)$ described by probabilities assigned to quantum states.

Next, we introduce the requirements on the first-order oracle.

\begin{assumption}\label{ass:first_order_oracle}
 {\bf Probabilistic first-order oracle.}	 
	Given a point $x\in\R^n$ and a constant $\alpha> 0$, the oracle computes $g(x,\alpha,\Xi^\prime)$ with realization $g(x,\alpha,\xi^\prime)$, a (random) estimate of the gradient $\nabla \phi(x)$.
	The third argument of $g$, $\Xi^\prime$, is a random variable (whose distribution may depend on $x$), with probability space  $(\Omega^\prime, \mathcal{F}_{\Omega^\prime}, P')$. 
	The function $g$ satisfies, for all $x\in\R^n$ and for all $\alpha>0$, 
	\begin{equation}\label{eq:first_order}
		{\mathbb P_{\Xi^\prime}\left (\|g(x,\alpha,\Xi^\prime)-\nabla \phi(x)\|\leq \max \{\varepsilon_g, \min\{\tau, \kappa \alpha\}\|g(x,\alpha,\Xi^\prime)\| \}\right)\geq 1-\delta,}
	\end{equation}
	for some non-negative parameters $(\varepsilon_g,\tau,\kappa,\delta)$, where $\delta\in [0,\frac{1}{2})$.
\end{assumption}

In \Cref{ass:first_order_oracle}, the input $\alpha$ to the oracle represents the step size in the quasi-Newton method that we will describe later. 
Thus, \Cref{ass:first_order_oracle} effectively says that as step sizes decrease, the oracle must be proportionally more accurate.
We note that \eqref{eq:first_order} says that with probability $1-\delta>\frac 1 2$, the error in the gradient estimate is bounded by the greater of a constant ($\varepsilon_g$) and a quantity that scales linearly with the magnitude of the realized gradient estimate.  With constant probability $\delta < \frac12$, we allow the gradient estimator to be arbitrarily bad. Observe that since there is a constant bias of  $\varepsilon_g$ in the gradient estimates, it is only possible to guarantee convergence to an approximate stationary point.

Just as with the zeroth-order oracle, $\Xi^\prime$ here represents all sources of randomness in the first-order oracle. 
For example, in expected risk minimization, $\Xi^\prime$ would represent the random minibatch used to compute the gradient estimate. In the VQA problem, the first-order oracle is implemented by taking differences of function estimates via \eqref{eq:magic_dd}, and thus $\Xi^\prime$ would represent the randomness in the function value estimates. 
 To simplify notation, we will often omit the dependence of $g(x,\alpha,\Xi^\prime)$ on $\alpha$ and write $g(x,\Xi^\prime)$.


\section{A CRN-Free Quasi-Newton Algorithm}
\label{sec:algo}
The method we propose is based on the Stochastic Adaptive Step Search Method (SASS) introduced in \cite{jin2021high}.
SASS employs the negative of the output of the probabilistic first-order oracle as a search direction.
Our method, on the other hand, effectively uses a (symmetric positive definite) L-BFGS matrix $B$; that is, our search direction also uses the negative of the output of the probabilistic first-order oracle but premultiplied by $B^{-1}$ (or $H$).  
As is common practice in L-BFGS implementations, the matrix $B$ or $H$ will never be stored explicitly, nor will direct matrix-vector multiplications involving $B^{-1}$ be performed. Rather, we will apply the standard two-loop recursion (see Algorithm 7.4 of \cite{NW}). 
The two-loop recursion only requires us to  store two sets of vectors, $S =\left[\begin{array}{lll}
	s_{1}, & \hdots, & s_{m}
\end{array}\right]$ and $Y =\left[\begin{array}{lll}
	y_{1}, & \hdots, & y_{m}
\end{array}\right]$. 
Each vector in $S$ represents a nonzero displacement $x_{k+1} - x_{k}$, whereas each vector in $Y$ represents a corresponding difference in gradient estimates $g(x_{k+1} ,\alpha_{k+1} ,\xi^\prime_{k+1} ) - g(x_{k},\alpha_{k},\xi^\prime_{k} )$. 

Employing gradient estimates that satisfy \Cref{ass:first_order_oracle} only for the L-BFGS matrix inverse $B^{-1}$ leads to obvious challenges. 
If the construction of $B^{-1}$ involves even one particularly poor estimate in $Y$ (as is certainly permitted by the probabilistic nature of \Cref{ass:first_order_oracle}), the quasi-Newton direction can potentially be poor compared to simply taking a negative gradient direction (i.e., replacing $B$ with the identity matrix). 
To address this issue, we deliberately monitor the eigenvalues of $B$ and, when necessary, remove vectors from $S$ and $Y$, so that the spectrum of $B$ lies within the interval $[\sigma_{\mathrm{lb}}^B,\sigma_{\mathrm{ub}}^B]$ for $\sigma_{\mathrm{lb}}^B > 0$ and $\sigma_{\mathrm{ub}}^B < \infty$.
 To be able to do this efficiently, we need an efficient method to check the largest and smallest eigenvalues of $B_k$ in the $k$th iteration. The procedure is described in \cite[Section 2.2]{brust2019dense} but is summarized in \Cref{alg:evalues}. 
 The procedure in \Cref{alg:evalues} exploits the compact representation of $B_k$. 
 
 \begin{algorithm}[h!]
\caption{Computing the extreme eigenvalues of L-BFGS matrix $B$}
\label{alg:evalues}
\textbf{Input: } $S =\left[\begin{array}{lll}
	s_{1}, & \hdots, & s_{m}
\end{array}\right], 
Y = \left[\begin{array}{lll}
	y_{1}, & \hdots, & y_{m}\end{array}\right]$, $c>0$.	

{\bf Set} $B_0 = cI$ and $\Psi = \left[\begin{array}{lll}
	B_{0}S & Y
\end{array}\right]$.

{\bf Compute} the thin QR-decomposition $\Psi = QR$ (i.e., $Q\in\R^{n\times 2m}$ and $R\in\R^{2m\times 2m}$). 

{\bf Compute} $\Phi = S^T Y$. Let $D$ and $L$ denote the diagonal and strict lower triangular parts of $\Phi$, respectively. 

{\bf Set} $\Gamma\gets -\left[\begin{array}{cc}
	S^{T} B_{0} S & L \\
	L^{T} & -D
\end{array}\right]^{-1}$.

{\bf Compute} the largest and smallest eigenvalues of $R\Gamma R^T$, $\sigma_1$ and $\sigma_{2m}$.

\textbf{Return: } Largest eigenvalue of $B$, $\sigma_l=\max\{\sigma_1+c,c\}$ and smallest eigenvalue of $B$, $\sigma_{s} = \min\{\sigma_{2m}+c,c\}$.


\end{algorithm} 

Using \Cref{alg:evalues}, one can obtain the smallest and largest eigenvalues of the Hessian approximation $B$  in $O(nm^2)$ flops, where $n$ is the dimension of the problem and $m$ is a small integer that represents the number of vector pairs L-BFGS keeps in memory. 
To ensure the eigenvalues of $B_k$  are bounded above and below, one may recursively remove $(s_i,y_i)$ pairs from $(S, Y)$ from the oldest to newest, until the spectral bounds $[\sigma_{\mathrm{lb}}^B,\sigma_{\mathrm{ub}}^B]$ are satisfied. This procedure is formally described in \Cref{alg:enforce}.  \ml{By the choice that $c\in [\sigma_{lb}^B,\sigma_{ub}^B]$, this procedure is guaranteed to terminate in $m$ steps. We note that due to \Cref{alg:enforce}, we always have that the eigenvalues of $B_k$ are deterministically bounded below by $\sigma_{\mathrm{lb}}^B$ and bounded above by $\sigma_{\mathrm{ub}}^B$. }

\begin{algorithm}[ht]
 	\caption{Enforcing $B$ has spectrum in $[\sigma_{\mathrm{lb}}^B,\sigma_{\mathrm{ub}}^B]$}
 	\label{alg:enforce}
{\bf Input:}  $S$, $Y$, \ml{bounds $0<\sigma_{\mathrm{lb}}^B \leq \sigma_{\mathrm{ub}}^B < \infty$, and $c\in[\sigma_{\mathrm{lb}}^B,\sigma_{\mathrm{ub}}^B]$}.

{\bf Set} $m = |S|$.

\For{$t = 1,\dots,m$}{
{\bf Run} \Cref{alg:evalues} with input $S,Y$, and $c$ to obtain $\sigma_l$, $\sigma_{s}$.\\
\uIf{$\sigma_l\leq\sigma_{\mathrm{ub}}^B$ and $\sigma_{s}\geq\sigma_{\mathrm{lb}}^B$}
{
{\bf Return:} $S, Y.$
}
\Else{
$S\gets S \setminus \{s_t\}, Y\gets Y\setminus \{y_t\}$.
}
}
\end{algorithm}

We are now prepared to introduce the quasi-Newton stochastic adaptive step search (Q-SASS) method, which is stated in \Cref{alg:aloe_quasi}.
In each iteration, Q-SASS computes a gradient estimate, $g_k$, at the current point $x_k$ using the first-order oracle. It then implicitly updates the quasi-Newton matrix $B_k^{-1}$, using \Cref{alg:enforce} to enforce that the spectrum of $B_k^{-1}$ is bounded. Next, Q-SASS computes the step $d_k = B_k^{-1}g_k$ (using the standard two-loop recursion).  Subsequently,  the candidate point $x_k^+ = x_k - \alpha_k d_k$ is computed, and a check is performed to see if $x_k^+$ appears to provide a sufficient decrease in the function value, using the zeroth-order oracle. If $x_k^+$ does appear to provide sufficient decrease, then $x_k^+$ is accepted as the new iterate,  and the step size $\alpha_k$ is increased in preparation for the next iteration. Otherwise, the iterate is unchanged and the step size is decreased.
We follow the convention that in iterations $k$ where $S,Y$ are empty, the two-loop recursion simply computes $B_k^{-1}g_k = B_0^{-1}g_k = c^{-1}g_k$.

	
  
	
	

\begin{algorithm}[!htpb]
 	\caption{~\textbf{ Quasi-Newton stochastic adaptive step search (Q-SASS)}}
	\label{alg:aloe_quasi}
	{\bf Input:}  Initial point $x_1$,  initial step size $\alpha_1>0$,  constants $\theta, \gamma \in (0,1)$, memory parameter $M$, upper and lower bounds $\sigma_{\mathrm{ub}}^B$, $\sigma_{\mathrm{lb}}^B>0$, $c\in [\sigma_{\mathrm{lb}}^B,\sigma_{\mathrm{ub}}^B]$, inner product tolerance $\theta_{\mathrm{ip}}$, error tolerance $\varepsilon_f$ (see \Cref{ass:zeroth_order_oracle}).  
	
	{\bf Initialize} $S = \emptyset$, $Y = \emptyset$, $x_{\mathrm{prev}} = \emptyset$, $g_{\mathrm{prev}}=\emptyset$.
	
	\For{$k=1, 2, \dots$}{
	{\bf Compute gradient approximation:} \\
	$g_k\gets g(x_k,\alpha_k,\xi_k^\prime)$ via the probabilistic first order oracle.
	
	{\bf Update quasi-Newton matrix: }\\
	\If{$x_{\mathrm{prev}}\neq\emptyset$}{
	$s_k \gets x_k - x_{\mathrm{prev}}$, $y_k \gets g_k - g_{\mathrm{prev}}.$\\
	\If{$\dotp{s_k,y_k}>\theta_{\mathrm{ip}}$}{
	\If{$|S|=M$}{
	$S \gets S\setminus \{s_1\},$ $Y \gets Y\setminus \{y_1\}$.
	}
        $S \gets \left[S, s_k\right],$ $Y \gets \left[Y, y_k\right]$.
	} 
	Apply \Cref{alg:enforce} with inputs $S$, $Y$, $c$, $\sigma_{\mathrm{ub}}^B$, $\sigma_{\mathrm{lb}}^B$.
	}
	{\bf Compute trial step: }\\
	Use two-loop recursion with $S$, $Y$, $B_0=cI$ to compute $d_k\gets B_k^{-1}g_k$.
	
	$x_k^+ \gets x_k - \alpha_k d_k$.
	
	{\bf Test for sufficient decrease: }\\
	$f_k \gets f(x_k,\xi_k)$, $f_k^+ \gets f(x_k^+,\xi_k^+)$ via the probabilistic zeroth order oracle.\\
	\uIf{$f_k^+ \leq f_k - \alpha_k\theta\dotp{d_k,g_k} + 2\varepsilon_f$}{
	{\bf \emph{Successful iteration}:}
	$x_{\mathrm{prev}} \gets x_k$, $g_{\mathrm{prev}}\gets g_k$, 
	$x_{k+1}\gets x_k^+$, $\alpha_{k+1}\gets \gamma^{-1}\alpha_k$.
	}
	\Else{
        {\bf \emph{Unsuccessful iteration}: } 
	$x_{k+1}\gets x_k$, $\alpha_{k+1}\gets\gamma\alpha_k$.
	}
	} 
\end{algorithm}

\section{Analysis and complexity result}\label{sec:analysis}

\subsection{General analysis framework}

Our analysis follows the framework in \cite{jin2021high}. 
Note that the iterations of \Cref{alg:aloe_quasi} define a stochastic process on random variables 
$\{({G}_k, X_k, D_k, A_k, {E}_k, {E}_k^+)\}$
with corresponding realizations 
 $(g_k, x_k, d_k, \alpha_k, e_k, e_k^+)$.
 While the first four of these realized quantities appear explicitly in the statement of \Cref{alg:aloe_quasi}, the latter two do not;
 instead, $e_k$ and $e_k^+$ are shorthand notations for $e(x_k,\xi_k) = |f(x_k,\xi_k)-\phi(x_k)|$ and $e(x_k^+,\xi_k^+)=|f(x_k^+,\xi_k^+) - \phi(x_k^+)|$, respectively. 

 The filtration of the stochastic process defined on 
 $\{({G}_k, X_k, D_k, A_k, {E}_k, {E}_k^+)\}$
 can be stated in terms of the exogenous random variables $M_k = \{\Xi_k, \Xi_k^+, \Xi_k^\prime\}$ with realizations $\{\xi_k, \xi_k^+, \xi_k^\prime\}$. 
 Recall that the randomness of $G_k$ is dictated by $\Xi_k^\prime$ in the first-order oracle, and the randomness of $E_k, E_k^+$ is dictated by $\Xi_k,\Xi_k^+$ in the zeroth-order oracle. 
 The filtration is 
 $\{{\cal F}_k:\, k\geq 0\}$, where ${\cal F}_k=\sigma (M_0, M_1, \ldots, M_k)$, the $\sigma$-algebra generated by $M_0,M_1,\ldots,M_k$.
 The probability measure associated with the stochastic process is derived from the exogenous random variables $M_k$. 
 
Given this stochastic process, one can define a stopping time.
To derive complexity results for \Cref{alg:aloe_quasi}, one has obvious choices of stopping times $T_\varepsilon$ parameterized by $\varepsilon>0$. 
In particular, given only \Cref{ass:Lipschitz}, the natural choice of a stopping time is 
\begin{equation}
\label{eq:nonconvex_stopping}
T_\varepsilon := \min\{k: \|\nabla \phi(X_k)\|\leq \varepsilon\},
\end{equation}
that is, the number of iterations until the norm of the objective gradient is sufficiently small. 
Similarly, if we additionally assume \Cref{ass:strongly_convex}, the associated stopping time is
\begin{equation}
    \label{eq:convex_stopping}
    T_\varepsilon := \min\{k: \phi(X_k)-\phi^*\leq\varepsilon\},
\end{equation}
that is, the number of iterations until the objective function value is sufficiently close to the global optimal. 

Similar to the analysis in \cite{jin2021high}, a major step of our analysis lies in proving that the following very generally stated \Cref{ass:alg_behave} holds for some carefully selected $Z_k$ (which provides a measure for the amount of potential progress that can be made at iteration $k$) in either nonconvex or strongly convex settings for the algorithm. For each $k \geq 0$, we let $Z_k \geq 0$ be a random variable that measures the progress of the algorithm at step $k$ and define 
$$Z_k = \begin{cases}
  \phi(X_k) - \phi^* & \mbox{if } \phi \mbox{ is convex}  \\
  \ln \left( \frac{\phi(X_k) - \phi^*}{\varepsilon} \right) & \mbox{if } \phi \mbox{ is nonconvex}.  
\end{cases}$$

 We now provide the following definition, which characterizes the iterations in which the errors made by the probabilistic first- and zeroth-order oracles are favorable. 
 \begin{definition}
 \label{def:true}
 The $k$th realized iteration of \Cref{alg:aloe_quasi} is \textbf{\emph{true}} provided both
 \begin{enumerate}
 \item $\|g_k - \nabla \phi(x_k)\| \leq \max\{\varepsilon_g,\min\{\tau, \kappa \alpha_k\} \|g_k\|\}$, 
 with $\varepsilon_g$, $\tau$, and $\kappa$ as in \Cref{ass:first_order_oracle}, and
 \item $e_k + e_k^+ \leq 2\varepsilon_f$, where $\varepsilon_f$ is as in \Cref{ass:zeroth_order_oracle}. 
 \end{enumerate}
 \end{definition}


Before introducing \Cref{ass:alg_behave}, 
we define two random variables, $I_k$ and $\Theta_k$, generated by the stochastic process underlying \Cref{alg:aloe_quasi}.

\begin{definition}\label{2RVs}
Recall the definition of a true iteration in \Cref{def:true} and the definition of a successful iteration in \Cref{alg:aloe_quasi}.
    Define the ${\cal F}_k$-measurable random variables
	$$I_k := \one\{\text{iteration $k$ is true}\} \quad \text{ and } \quad \Theta_k := \one\{\text{iteration $k$ is successful}\}.$$
\end{definition}

\begin{assumption}
	 \label{ass:alg_behave}
	There exist 
	\begin{itemize}
	    \item a constant $\bar{\alpha}>0$,
	    \item a nondecreasing function $h: \mathbb{R}_+ \rightarrow \mathbb{R}_+$, 
	    \item a function $r: \R^2 \to \R$ which is non-decreasing and concave in its second argument,
	    \item and a constant $p \in ( \frac12+ \frac{{r(\varepsilon_f, 2\varepsilon_f)}}{h(\bar{\alpha})}, 1]$,
	\end{itemize} 
	such that the following five properties hold for all $k<T_{\varepsilon}$:
	\begin{enumerate}
	\item $h(\bar{\alpha})>{\frac{r(\varepsilon_f, 2\varepsilon_f)}{p-\frac 1 2}}$. (Progress is large enough relative to the error.)
	\item $\mathbb{P}\left(I_{k}=1 \mid \mathcal{F}_{k-1}\right) \geq p,$ for all $k$. (Each iteration is true with probability at least $p$, regardless of the past.)
	\item If $I_{k} \Theta_{k}=1$, then $Z_{k+1}\leq Z_k-h(A_k)+{r(\varepsilon_f, 2\varepsilon_f)}$. (True, successful iterations make progress.)
	\item If $A_{k} \leq \bar{\alpha}$ and $I_{k}=1$, then $\Theta_{k}=1$. (Any small, true iteration is successful.)
	\item $Z_{k+1} \leq Z_{k}+{r(\varepsilon_f, {E}_k+{E}_k^+)}$ for all $k$. (A bound on the damage to the progress in each iteration.)
	\end{enumerate}
\end{assumption}

\Cref{ass:alg_behave} is nontrivial and requires some discussion, which we will provide momentarily.
First, we note that so long as
 \Cref{ass:alg_behave} holds, then we can demonstrate, as in \cite{jin2021high}, a high probability bound on $T_\varepsilon$:

\begin{theorem}
  [Theorem~3.8 in \cite{jin2021high}]
\label{thm:general_complexity}
Let \Cref{ass:zeroth_order_oracle}, and \Cref{ass:first_order_oracle} hold, and suppose \Cref{ass:alg_behave} holds for a given ${\cal F}_k$-measurable random variable $Z_k$.
Let $\bar\alpha,h,r$, and $p$ be as in \Cref{ass:alg_behave}, 
let $\varepsilon_f$ be as in \Cref{ass:zeroth_order_oracle}, let $\alpha_0$,  $\gamma$ be taken from the initialization step of \Cref{alg:aloe_quasi}, 
and let $\nu_r,b_r$ be the parameters associated with the subexponential random variable $r(\varepsilon_f, E_k+E_K^+).$\footnote{A subexponential random variable $Y$ having parameters $\nu,b$ means that $\mathbb{E}\left[\exp\left(\lambda(Y-\mathbb{E}\left[Y]\right])\right) \right] \leq \exp\left(\frac{\lambda^2\nu^2}{2}\right)$ for all $\lambda$ satisfying $|\lambda|<\frac{1}{b}$. Compare with \cref{eq:zero_order2}.}
For any $s\geq 0$, denote $p_{\ell} = \frac12 + \displaystyle\frac{r(\varepsilon_f,2\varepsilon_f)+s}{h(\bar\alpha)}$. Then for any $\hat{p}\in(p_{\ell},p)$ and for any
$$t\geq \displaystyle\frac{\frac{Z_0}{h(\bar\alpha)} + \max\left\{-\frac{\ln{\alpha_0}-\ln{\bar\alpha}}{2\ln{\gamma}},0\right\}}
{\hat{p}-p_{\ell}},$$
it holds that
$$\mathbb{P}\left[T_\varepsilon \leq t \right] \geq 1 - \exp\left(-\displaystyle\frac{(p-\hat{p})^2}{2p^2}t\right) - \exp\left(-\min\left\{\frac{s^2t}{2\nu_r^2},\frac{st}{2b_r}\right\}\right).$$
\end{theorem}

Put simply, \Cref{thm:general_complexity} essentially says that the probability of the stopping time $T_\varepsilon$ (for both \eqref{eq:nonconvex_stopping} and \eqref{eq:convex_stopping}) being greater than $t$ decays exponentially in $t$, for $t$ sufficiently large (on the order of $\frac{Z_0}{h(\bar\alpha)})$. \ml{We note that for \Cref{ass:alg_behave} to hold, there needs to be a lower bound for the $\varepsilon$ in the stopping time $T_\varepsilon$. The lower bound for $\varepsilon$ is dictated by $\varepsilon_f$ and $\varepsilon_g$, which essentially is the amount of ``bias" in the zeroth- and first-order oracle. The specific expression for the lower bound depends on the function class (strongly convex, non-convex, etc.), and we will state the explicit form in the subsequent subsections. }
 
Thus, we have only to verify that \Cref{ass:alg_behave} holds for \Cref{alg:aloe_quasi}. We will verify that this assumption holds in the nonconvex and strongly convex settings in \Cref{sec:nonconvex} and \Cref{sec:strongly_convex},  respectively. This will lead to specific high probability iteration complexity bounds for \Cref{alg:aloe_quasi} in each setting.

We now pause to unpack what \Cref{ass:alg_behave} actually says, now that the remainder of the analysis is effectively dedicated to showing \Cref{ass:alg_behave} holds in various settings. 
Informally speaking, and in the context of \Cref{alg:aloe_quasi} , $\bar\alpha$ should be viewed as a sufficiently small step size (in the deterministic setting, $\bar\alpha$ is on the order of $\frac 1 L$ for an $L$-smooth function), and $h(\alpha)$ should be viewed as a quantity related to the decrease in the true function value one should expect in a deterministic line search method when the step size $\alpha$ is used. $r(\varepsilon_f, e_k + e_k^+)$ is a bound on the ``damage" to the progress of the algorithm when the errors in the zeroth-order oracle at $x_k$ and $x_k^+$ are $e_k$ and $e_k^+$, respectively.
With this interpretation, \Cref{ass:alg_behave} states that in every iteration until the stopping time is reached, 
\begin{enumerate}
    \item The progress when using the sufficiently small step size is sufficiently large compared to the damage. 
    \item Iterations are true (i.e., the oracles give reasonable outputs) with some sufficiently large constant probability $p$, regardless of the past; 
    \item If an iteration is true and successful, then the progress should resemble, up to some additive error term, the progress observed in a deterministic line search method; 
    \item If the step size is sufficiently small and the iteration is true, then success is guaranteed; and
    \item The damage to the progress of the algorithm is limited by a quantity that scales with the errors in the zeroth-order oracle. 
\end{enumerate}

We will now demonstrate complexity results for \Cref{alg:aloe_quasi} under \Cref{ass:Lipschitz} (that is, the general \emph{nonconvex case}) in \Cref{sec:nonconvex} and then additionally under \Cref{ass:strongly_convex} (that is, the \emph{strongly convex case)} in \Cref{sec:strongly_convex}. 

\subsection{Nonconvex functions}\label{sec:nonconvex}
We begin with the case where $\phi$ is nonconvex; that is, we begin by considering the stopping time $T_\varepsilon$ in \cref{eq:nonconvex_stopping}. We will show that the algorithm obtains an $\varepsilon$-stationary point in $O(\frac{1}{\varepsilon^2})$ iterations with overwhelmingly high probability.

Before directly analyzing the stochastic process, we note that due to \Cref{alg:enforce}, we always have that the eigenvalues of $B_k$ are deterministically bounded below by $\sigma_{\mathrm{lb}}^B$ and bounded above by $\sigma_{\mathrm{ub}}^B$. 
As a result, we can prove the following result concerning realizations of \Cref{alg:aloe_quasi}. 

\begin{proposition}\label{prop:innerproduct}
Let \Cref{ass:Lipschitz} hold, $\sigma_{\mathrm{ub}}=\left(\sigma_{\mathrm{lb}}^B\right)^{-1}$, and $\sigma_{\mathrm{lb}}=\left(\sigma_{\mathrm{ub}}^B\right)^{-1}$, \ml{where $\sigma_{ub}^B\geq \sigma_{lb}^B>0$}. 
For all $k$, 
		\begin{equation}
		\label{eq:innerproduct}
		d_{k}^{T} g_{k} \geq \frac{\sigma_{\mathrm{lb}}}{\sigma_{\mathrm{ub}}}{\left\|d_{k}\right\|\left\|g_{k}\right\|}.
		\end{equation}
\end{proposition}
\begin{proof}
Due to \Cref{alg:enforce}, we have that, for all $k$, $d_k = B_k^{-1}g_k$ satisfies
\begin{equation}\label{eq:ebounds} 
    \sigma_{\mathrm{lb}}\left\|g_{k}\right\| \leq\left\|d_{k}\right\| \leq \sigma_{\mathrm{ub}}\left\|g_{k}\right\|.
\end{equation}
Thus, \cref{eq:innerproduct} is satisfied, because
	$$\frac{\sigma_{\mathrm{lb}}\norm{d_k}\norm{g_k}}{\sigma_{\mathrm{ub}}}\leq \sigma_{\mathrm{lb}}\norm{g_k}^2 \leq d_k^Tg_k.$$
 \qed
\end{proof}

One may have noticed that because of the presence of the constant $\varepsilon_f$ in \Cref{ass:zeroth_order_oracle} (i.e., noise in function estimates) and also because of the presence of the constant $\varepsilon_g$ in \Cref{ass:first_order_oracle} (i.e., bias in gradient estimates), it is unreasonable to expect \Cref{alg:aloe_quasi} to attain arbitrary accuracy $\varepsilon>0$. Instead, the size of the convergence neighborhood is determined by the amount of bias in the oracles.
The following inequality provides
the accuracy achievable in light of these limitations.
\begin{inequality}
\label{ass:inequality}
The accuracy level $\varepsilon$ satisfies
    	\begin{equation*}
		\varepsilon > \max\left\{\frac{\varepsilon_g}{\eta}, \,\sqrt{\frac{4\varepsilon_f}{M_1\bar\alpha(p - \frac12)}  } \right\},
		 ~\text{for some}~ 
		\eta\in \left(0, \, \displaystyle\frac{1-\theta K}{1+(1-\theta)K}\right) ~\text{and}~ p>\frac 1 2,
	\end{equation*}
where $K=\frac{\sigma_{\mathrm{lb}}}{\sigma_{\mathrm{ub}}}$
	and $M_1=\min\left\{\displaystyle\frac{\sigma_{\mathrm{lb}}\theta}{(1+\tau)^2}, \, \sigma_{\mathrm{lb}}\theta(1-\eta)^2\right\}.$
	Here, $\varepsilon_f$ is from \Cref{ass:zeroth_order_oracle}, 
	$\varepsilon_g$, $\tau$ and $\kappa$ are from \Cref{ass:first_order_oracle}, 
	$\bar\alpha$ and $p$ are from \Cref{ass:alg_behave} 
 (we will specify their values in \Cref{prop:ass_nonconvex}), 
 $\sigma_{\mathrm{lb}}$, $\sigma_{\mathrm{ub}}$ are as defined in \Cref{prop:innerproduct}, and $\theta$ is from the initialization step of \Cref{alg:aloe_quasi}. 
\end{inequality}

We now  show \Cref{ass:alg_behave} is indeed satisfied by \Cref{alg:aloe_quasi} when applied to nonconvex functions.

\begin{proposition}[\Cref{ass:alg_behave} holds in the nonconvex case]
	\label{prop:ass_nonconvex} 
	Let the stopping time $T_\varepsilon$ be given by \cref{eq:nonconvex_stopping}, and define 
	$$Z_k := \phi(X_k) - \phi^*.$$
 Suppose 
 \Cref{ass:Lipschitz}, \Crefrange{ass:zeroth_order_oracle}{ass:first_order_oracle} and \Cref{ass:inequality} hold. 
	Then, \Cref{ass:alg_behave} holds with 
	\begin{itemize}
	    \item $\bar{\alpha} = \min\left\{\frac{2(1- \theta)\sigma_{\mathrm{lb}}}{\left({2\kappa+ L \sigma_{\mathrm{ub}}}\right)\sigma_{\mathrm{ub}}},
 \frac{2\left((1-\theta)(1-\eta)\frac{\sigma_{\mathrm{lb}}}{\sigma_{\mathrm{ub}}}-\eta\right)}{L\sigma_{\mathrm{ub}}(1-\eta)} \right\}$.
 		\item $h(\alpha) = M_1\alpha\varepsilon^2$, where $M_1=\min\left\{\frac{\sigma_{\mathrm{lb}}\theta}{(1+\tau)^2}, \,\sigma_{\mathrm{lb}}\theta(1-\eta)^2\right\}$. 
 	    \item $r(\varepsilon_{f}, E_k+E_k^+)= 2\varepsilon_{f}+E_k+E_k^+$.
		\item $p = 1 - \delta$ (for bounded noise), or $p = 1-\delta -\exp\left(-\min\{\frac{u^2}{2\nu^2},\frac{u}{2b}\}\right)$ otherwise. Here $u = \inf_x \{\varepsilon_f - \E[e(x)]\}$. Furthermore, assume $p > \frac12$. 
		
	\end{itemize}
\end{proposition}

For ease of reading, the proof of \Cref{prop:ass_nonconvex} is deferred to \Cref{sec:propncproof} in the Appendix. 

Combining \Cref{thm:general_complexity} with \Cref{prop:ass_nonconvex} and noting that $r(\varepsilon_f,E_k+E_k^+) = 2\varepsilon_f + E_k + E_k^+$ is clearly a subexponential random variable with parameters $(\nu_r,b_r) = (2\nu,2b)$, where $\nu$ and $b$ are from \Cref{ass:zeroth_order_oracle},
we immediately obtain the following explicit iteration complexity bound for \Cref{alg:aloe_quasi} in the nonconvex case. 

\begin{theorem}
	[Iteration complexity for \Cref{alg:aloe_quasi} for nonconvex functions]
	\label{thm:nonconvex_complexity}
	Let the stopping time $T_\varepsilon$ be given by \cref{eq:nonconvex_stopping}, and assume	\Cref{ass:Lipschitz}, \Crefrange{ass:zeroth_order_oracle}{ass:first_order_oracle}, and that $\varepsilon$ satisfies \Cref{ass:inequality}. 
	Let $s\geq 0$, and denote $p_{\ell} = \frac12 + \displaystyle\frac{4\varepsilon_{f}+s}{M_1\bar\alpha\varepsilon^2}$.
	Then, for any $\hat{p}\in(p_{\ell},p)$ and any
$$t\geq \displaystyle\frac{\frac{\phi(x_0)-\phi^*}{M_1\bar\alpha\varepsilon^2} + \max\left\{-\frac{\ln{\alpha_0}-\ln{\bar\alpha}}{2\ln{\gamma}},0\right\}}
{\hat{p}-p_{\ell}},$$
 we have that
	$$\P\left(T_\varepsilon \leq t\right) \geq 1 - \exp\left(-\frac{(p-\hat{p})^2}{2p^2}t\right) - \exp{\left(-\min\left\{\frac{s^2t}{{8\nu^2}},\frac{st}{{4b}}\right\}\right)},$$ 
where  $\bar\alpha$, $M_1,$ and $p$ are as in \Cref{prop:ass_nonconvex}.
\end{theorem}

    \Cref{thm:nonconvex_complexity} essentially shows that the number of iterations required by \Cref{alg:aloe_quasi} to attain an $\varepsilon$-stationary point is in $O(\frac{1}{\varepsilon^2})$ with overwhelmingly high probability, and furthermore, the stopping time itself is a subexponential random variable. 
	
\subsection{Strongly convex functions}\label{sec:strongly_convex}
We now consider the strongly convex case, with the stopping time $T_\varepsilon$ defined in \cref{eq:convex_stopping}. We will show  that \Cref{alg:aloe_quasi} converges linearly to a neighborhood of optimality with overwhelmingly high probability.

Analogously to \Cref{sec:nonconvex}, we begin by providing a lower bound on the accuracy achievable by the algorithm with respect to the noise and bias in the oracles. 

\begin{inequality}
    \label{ass:inequality_sc}
The accuracy level $\varepsilon$ satisfies 
 {\tiny
	\begin{align*}
		\varepsilon > \max\left\{\frac{\varepsilon_g^2}{2\beta\eta^2},\, 4\varepsilon_f,\frac{4\varepsilon_f}{\left( 1-\min\left\{\frac{1}{(1+\tau)^2},(1-\eta)^2\right\} \sigma_{\mathrm{lb}}\theta \beta \cdot \min\left\{\frac{2(1- \theta)K}{\left({2\kappa+ L \sigma_{\mathrm{ub}}}\right)},
			\frac{2\left((1-\theta)(1-\eta)K-\eta\right)}{L\sigma_{\mathrm{ub}}(1-\eta)} \right\} \right)^{\frac 1 2-p}-1}
		\right\},
		\end{align*}
  }
		for some $\eta\in \left(0,\,\displaystyle\frac{1-\theta K}{1+(1-\theta)K}\right)$, where $K=\frac{\sigma_{\mathrm{lb}}}{\sigma_{\mathrm{ub}}}$ and $p>\frac 1 2$. 
		Here, $\beta$ is the strong convexity parameter in \Cref{ass:strongly_convex}.
\end{inequality}
Now, as in \Cref{sec:nonconvex}, we show that in the strongly convex setting, under our assumptions, \Cref{ass:alg_behave} holds for a particular choice of $\bar\alpha$, $h$, $r,$ and $p$. 

\begin{proposition}
[\Cref{ass:alg_behave} holds for \Cref{alg:aloe_quasi} in the strongly convex case]
	\label{prop:ass_strongly_convex}
	Let the stopping time be given by \cref{eq:convex_stopping}, and define
	$$Z_k = \ln \left( \frac{\phi(X_k) - \phi^*}{\varepsilon} \right).$$
Suppose \Crefrange{ass:Lipschitz}{ass:first_order_oracle} and \Cref{ass:inequality_sc} hold. 
Then, \Cref{ass:alg_behave} holds with the specific choices 
	\begin{itemize}
		\item $\bar{\alpha} = \min\left\{\frac{2(1- \theta)\sigma_{\mathrm{lb}}}{\left({2\kappa+ L \sigma_{\mathrm{ub}}}\right)\sigma_{\mathrm{ub}}},
	\frac{2\left((1-\theta)(1-\eta)\frac{\sigma_{\mathrm{lb}}}{\sigma_{\mathrm{ub}}}-\eta\right)}{L\sigma_{\mathrm{ub}}(1-\eta)} \right\}$;
	\item $h(\alpha) = \min\left\{ -\ln\left(1-\frac{\alpha\sigma_{\mathrm{lb}}\theta\beta}{(1+\tau)^2}\right), -\ln\left(1-\alpha\sigma_{\mathrm{lb}}\beta\theta(1-\eta)^2 \right)\right\}$;
			\item $r(\varepsilon_f, E_k+E_k^+) = \ln\left(1 + \frac{2\varepsilon_f + E_k+E_k^+}{\varepsilon} \right)$; 
		\item $p = 1 - \delta$ (for bounded noise), or $p = 1-\delta -\exp\left(-\min\{\frac{u^2}{2\nu^2},\frac{u}{2b}\}\right)$ otherwise. Here $u = \inf_x \{\varepsilon_f - \E[e(x)]\}$.
	\end{itemize}
\end{proposition}

The proof of \Cref{prop:ass_strongly_convex} is analogous to that of \Cref{prop:ass_nonconvex}. Due to space constraints, the proof is omitted. 

In the strongly convex case, as opposed to the nonconvex case, the subexponential parameters $(\nu_r,b_r)$ of 
$r(\varepsilon_f, E_k+E_k^+) = \ln\left(1 + \frac{2\varepsilon_f + E_k+E_k^+}{\varepsilon} \right)$
are less trivial to derive. 
However, \cite[Proposition 3]{jin2021high} demonstrates that these subexponential parameters are
\begin{equation}\label{eq:sub_para}{\nu_r = b_r = 4e^2\max\left\{\frac{2\nu}{\varepsilon}, \frac{2b}{\varepsilon}\right\} + 4e\left(1+\frac{4\varepsilon_f}{\varepsilon}\right)}.
\end{equation}
Combining this result with \Cref{thm:general_complexity} and \Cref{prop:ass_strongly_convex}, we arrive at a complexity result for the strongly convex case.

\begin{theorem}[Iteration complexity for \Cref{alg:aloe_quasi} for strongly convex functions]
\label{thm:sc_complexity}
Let the stopping time $T_\varepsilon$ be given by \cref{eq:convex_stopping}, and suppose  \Crefrange{ass:Lipschitz}{ass:first_order_oracle} and \Cref{ass:inequality_sc} hold. 
Let $s\geq 0$, and denote 
$$p_\ell = \frac12 + \frac{\ln\left(1 + \frac{4\varepsilon_f}{\varepsilon} \right)+s}{h(\bar\alpha)}.$$ 
Then, for any $\hat{p}\in(p_\ell,p)$ and for any
$$t \geq \displaystyle\frac{\frac{1}{h(\bar\alpha)}\ln\left(\frac{\phi(x_0)-\phi^*}{\varepsilon}\right) + \max\left\{-\frac{\ln \alpha_0 - \ln \bar \alpha}{2\ln \gamma}, 0\right\}}{\hat{p}-p_\ell},$$
we have that
		\begin{align*}
			\P\left(T_\varepsilon \leq t\right) \geq 1 
			- &\exp\left(-\frac{(p-\hat{p})^2}{2p^2}t\right) 
			- \exp\left(-\min\left\{\frac{s^2t}{2\nu_r^2},\frac{st}{2b_r}\right\}\right),
		\end{align*}
  with $\bar\alpha$, $h$, and $p$ as in \Cref{prop:ass_strongly_convex}, and $\nu_r$, $b_r$ as in \cref{eq:sub_para}.
\end{theorem}

\Cref{thm:sc_complexity} shows that \Cref{alg:aloe_quasi} converges linearly to a particular $\varepsilon$-neighborhood of optimality with overwhelmingly high probability, and furthermore the stopping time itself is a subexponential random variable.

\section{Numerical Experiments}\label{experiment}
To examine the benefits of our proposed quasi-Newton enhancement to SASS, we implemented \Cref{alg:aloe_quasi} in Python. Note that both SASS and Q-SASS do not need to leverage CRNs. 
We remark on a \emph{practical} augmentation we made to \Cref{alg:aloe_quasi}. 
To ease the analysis in \Cref{sec:analysis}, we assumed via \Cref{ass:zeroth_order_oracle} that $\varepsilon_f$ is constant across all iterations. 
However, since using a fixed small $\varepsilon_f$ over all iterations makes unnecessarily high accuracy demands on the zeroth-order oracle, it is of practical interest to try to choose them adaptively.
Thus, in our implementation, instead of using a constant $\varepsilon_{f}$,
an adaptive choice is made. 
In particular, in each iteration, we replace the constant value $\varepsilon_f$ in \Cref{ass:zeroth_order_oracle} and \Cref{alg:aloe_quasi} with 
 $$\varepsilon_{f,k}:= \max\left\{\varepsilon_{f}, \frac{1}{100}\alpha_k\theta g_k^Td_k \right\},$$ where $\varepsilon_f$ is chosen to be a fixed small number related to the desired target accuracy and the second term is a fraction of sufficient decrease imposed in \Cref{alg:aloe_quasi}. 
 If the fraction of sufficient decrease is large, the right decision of whether to accept the candidate point can still be made, even with a higher noise level in the function estimates.
 We note that this is a simpler form of adaptivity than can be found for QN methods that employ CRNs for zeroth-order stochastic oracles \cite{bollapragada2021adaptive}.

 We tested our implementation on synthetic problems derived from the \texttt{CUTEst} problem set \cite{cutest}. In particular, we use the same subset of 27 unconstrained problems\footnote{\cite{cartis2019improving} actually considered 30 problems, but we omit problems \texttt{BDEXP, SINEALI,} and \texttt{RAYBENDL}, because they are bound-constrained problems; removing the bound constraints results in problems that are unbounded below.
 }  considered in Appendix D of \cite{cartis2019improving}. 
 We additionally considered a subset of the quantum chemistry problems  in \cite{menickelly2022latency}; 
 we refer readers to that paper for technical details concerning the ansatz $\psi$ employed in \cref{eq:vqe}. Our code is available upon request.
 
\subsection{Synthetic problems with additive and multiplicative noise}\label{sec:synthetic}
To turn the subset of \texttt{CUTEst} problems into stochastic optimization problems, we artificially added noise to the problems. 
We considered two types of noise: \emph{additive noise} and \emph{multiplicative noise}. 
In either setting, 
we let $\xi\sim\mathcal{N}(0,1)$ (that is, $\xi$ follows a standard normal distribution)
and 
$\xi^\prime\sim\mathcal{N}(0_n,I_n)$ (that is, $\xi^\prime$ follows a multivariate Gaussian distribution with zero mean and identity covariance).

In the additive noise setting, 
we define
\begin{equation}
\label{eq:additive_noise}
f(x;\xi) = \phi(x) + \xi \quad \text{ and } \quad 
g(x,\xi^\prime) = \nabla\phi(x) + \xi^\prime.
\end{equation}
In the multiplicative noise setting, we define
\begin{equation}
\label{eq:multiplicative_noise}
f(x,\xi)=\left(1+\frac{\xi}{100}\right)\phi(x) \quad \text{ and } \quad
g(x,\xi^\prime) = \left(\one_n + \frac{1}{100} \xi^\prime\right)\odot\nabla\phi(x),
\end{equation}
where $\odot$ denotes the entrywise product, and $\one_n$ denotes the all-ones vector of length $n$. In both the additive and multiplicative noise settings, $f(x,\xi)$ and $g(x,\xi^\prime)$ are unbiased estimators of $\phi(x)$ and $\nabla\phi(x)$, respectively. 

The oracles for both noise settings can be defined as the average of a number of random realizations. 
In particular, at a point $x$ in the $k$th iteration of \Cref{alg:aloe_quasi}, we consider a sample average of $N_{f,k}$ many observations of $f$ as the zeroth-order oracle  and a sample average of $N_{g,k}$ many observations of $g$ as the first-order oracle, where
\begin{equation}
\label{eq:nfk_ngk}
N_{f,k}= \frac{\mathbb{V}(f(x,\Xi))}{\varepsilon_{f,k}^2} ~\text{ and }~
N_{g,k}=\frac{\mathbb{V}(g\left(x,\Xi^{\prime}\right))}{\delta\varepsilon_{g,k}^2},
\end{equation}
with $\delta$ as in \eqref{eq:first_order}, $$\varepsilon_{g,k} \triangleq \max\{\varepsilon_g,\min\{\tau,\kappa\alpha_k\}\|g_{k-1}\|\},$$
and $\mathbb{V}(\cdot)$ denoting (an estimate of) the variance of the argument estimator. 
We note that the choice of $N_{g,k}$ in \cref{eq:nfk_ngk} is motivated by the Chebyshev inequality, since if $\mathbb{V}(\cdot)$ were replaced with the true variance and $\|g_{k-1}\|$ were replaced by $\|g_{k}\|$, this sample size would guarantee that the oracle assumption is satisfied. 

The choice of $N_{f,k}$ is designed to approximately satisfy \Cref{ass:zeroth_order_oracle} with the adaptive $\varepsilon_{f,k}$.
Let $e(x,\xi)=\abs{\frac{1}{N_{f,k}}\sum_{i=1}^{N_{f,k}}f(x,\xi_i)-\phi(x)}$. To satisfy \Cref{ass:zeroth_order_oracle}, we need $\EE(e(x,\Xi))\leq \varepsilon_{f,k}$. 
Provided $N_{f,k}$ satisfies \cref{eq:nfk_ngk} or equivalently 
$\varepsilon_{f,k}= \sqrt{\frac{\mathbb{V}(f(x,\Xi))}{N_{f,k}}},$ 
we have that
$$\varepsilon_{f,k}= \sqrt{\frac{\EE[(f(x,\Xi)-\phi(x))^2]}{N_{f,k}}}
= \sqrt{{\EE[e(x,\Xi)^2]}}
\geq \sqrt{(\EE[e(x,\Xi)])^2}
={\EE[e(x,\Xi)]}.$$
By replacing the variance with sample variance in 
\cref{eq:nfk_ngk}, 
we can choose $N_{f,k}$ so that the estimator  approximately satisfies
\Cref{ass:zeroth_order_oracle}.
In our implementation, we implicitly assume that the gradient and variances do not change too rapidly between consecutive iterations, so we take the gradient and the sample variances of the last iteration as reasonable estimates. 

We now discuss choices of various parameter values in \Cref{alg:aloe_quasi}. 
In the initialization step of \Cref{alg:aloe_quasi}, we choose $\theta = 0.2, \gamma =0.8$, and $\alpha_0 = 1$.
Although we do not allow access to (estimates of) the true Hessian $\nabla^2 \phi(x)$ in any of our methods, for the sake of our experiments we assume that we have access to $\nabla^2 \phi(x_0)$ so that we may set
$\sigma_{\mathrm{ub}}^B=\max\{\|\nabla^2 \phi(x_0)\|,10^4\}$,\footnote{We note that in many applications, some reasonable upper bound $\sigma_{\mathrm{ub}}^B$ on the objective Hessian norm $\|\nabla^2 \phi(x)\|$ over a neighborhood of $x_0$ can often be provided by a domain expert, and our artificial $\sigma_{\mathrm{ub}}^B$ used in these tests is a proxy for such a bound.} where the matrix norm is the operator norm.
In turn, we set $\sigma_{\mathrm{lb}}^B = 1/\sigma_{\mathrm{ub}}^B$. 
We choose $\delta=0.1$ and $\tau=10$ in our definition of $N_{g,k}$ in \cref{eq:nfk_ngk}.
Given a target gradient norm $\bar\varepsilon$, we define a secondary parameter $\mu$ and set $\varepsilon_g = \mu\bar\varepsilon$ and $\varepsilon_f = \varepsilon_g^2$. 
In our experiments, we will demonstrate the effect of varying $\kappa$ and $\mu$, both of which control the precision of the oracles;
however, unless otherwise stated, we set $\kappa=1$ and $\mu=0.01$ as default values.

We first compare our implementation of Q-SASS against a natural baseline, SASS, which was studied in \cite{jin2021high}. 
SASS can be viewed as a special case of Q-SASS wherein we set the memory parameter, $M$, to be 0.
As a default setting for Q-SASS, we choose $M=10$.
We compare SASS and Q-SASS via performance profiles \cite{dolan2002benchmarking}.
By providing 30 different random seeds at the start of an optimization run, 
we yield $30\times 27 = 810$ instances represented in each performance profile. 
Motivated by the same latency concerns discussed in \cite{menickelly2022latency}, we measure the performance of solvers in two separate metrics: the total number of iterations to reach a given stopping time and the total number of samples observed to reach a given stopping time. 
In the latter metric we define ``a sample"  as an observation of either $f(x;\xi),$ or $g(x;\xi')$. 
Although in this paper we are mainly characterizing the iteration complexity of the algorithm, one may use the high-probability step size lower-bound result as in \cite{jinhigh} to obtain a high-probability upper bound on the sample complexity as well.
 The stopping time for each \texttt{CUTEst} problem is defined as \ml{the first iteration $k$ with $\norm{\nabla\phi(x_k)}\leq 10^{-3}\|\nabla \phi(x_0)\|$.}
 If a solver did not hit the stopping time within an appropriate budget of a maximum of $\min\{30000,500n\}$ iterations for the iteration metric, or a maximum of $10^{20}$ observed samples for the sample metric, then the performance metric of the algorithm for that problem is set to $\infty$.

The performance profiles for these initial additive noise and multiplicative noise tests are shown in \Cref{fig:additive} and \Cref{fig:multiplicative}, respectively. 
\ml{The way to read these plots is as follows. There is one curve for each algorithm. Consider a given curve and a point on that curve. Suppose the $x$-value of the point is $r$. Then, its corresponding $y$-value is the percentage of all instances that the algorithm was able to solve (as defined above) while using at most a factor $r$ times the cost (e.g. number of iterations or samples) of the best algorithm for that instance.}
We see that Q-SASS has an advantage over the baseline SASS algorithm  in terms of both iteration metric and  sample metric with both types of noise.

\begin{figure}[h!]\centering
	\begin{subfigure}{.49\textwidth}
		\centering
		\includegraphics[trim=3 -20 20 0, clip, width=\linewidth]{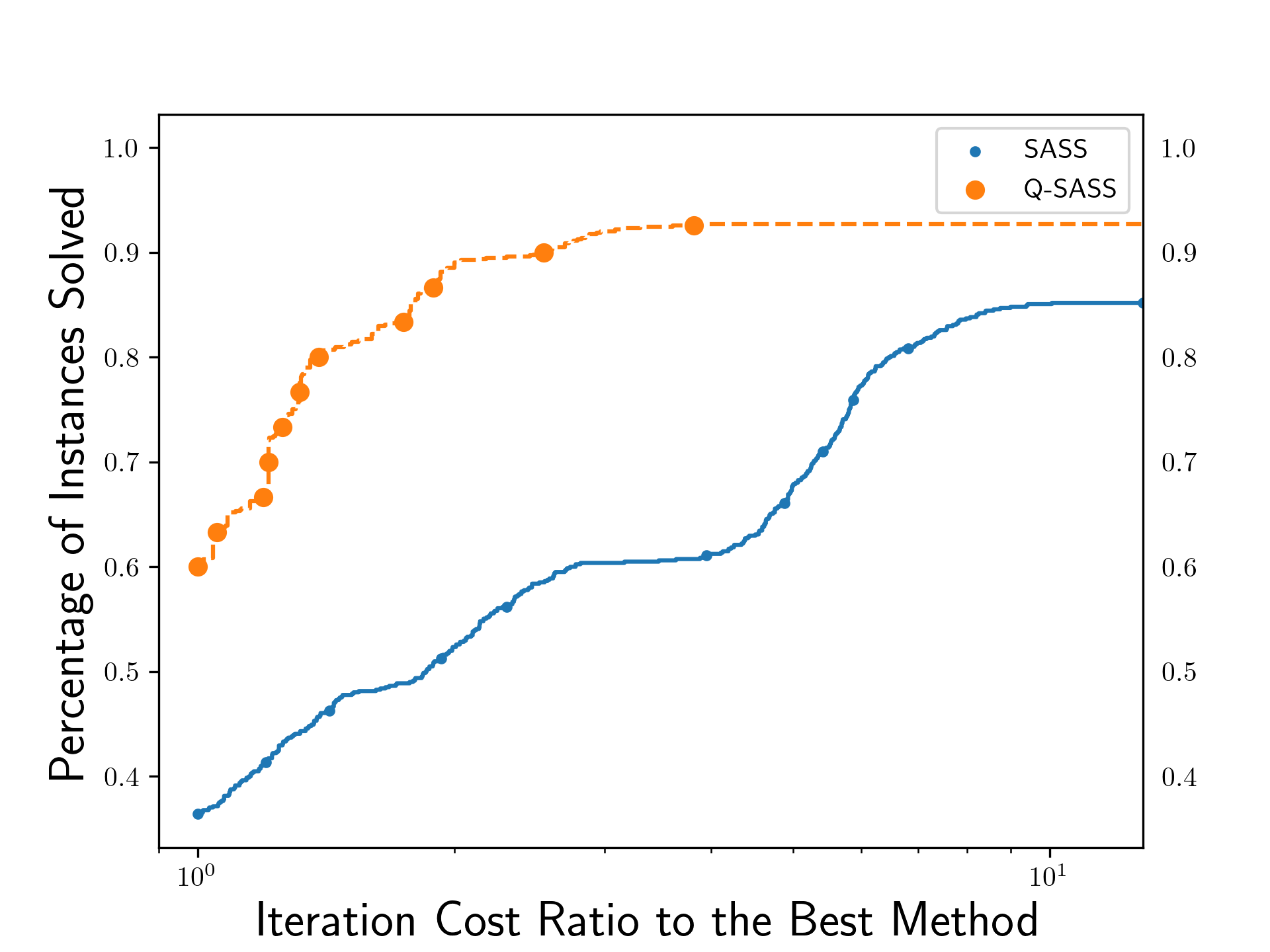}
	\end{subfigure}
	\hfill
	\begin{subfigure}{.49\textwidth}
	\centering
	\includegraphics[trim=3 -20 20 0, clip, width=\linewidth]{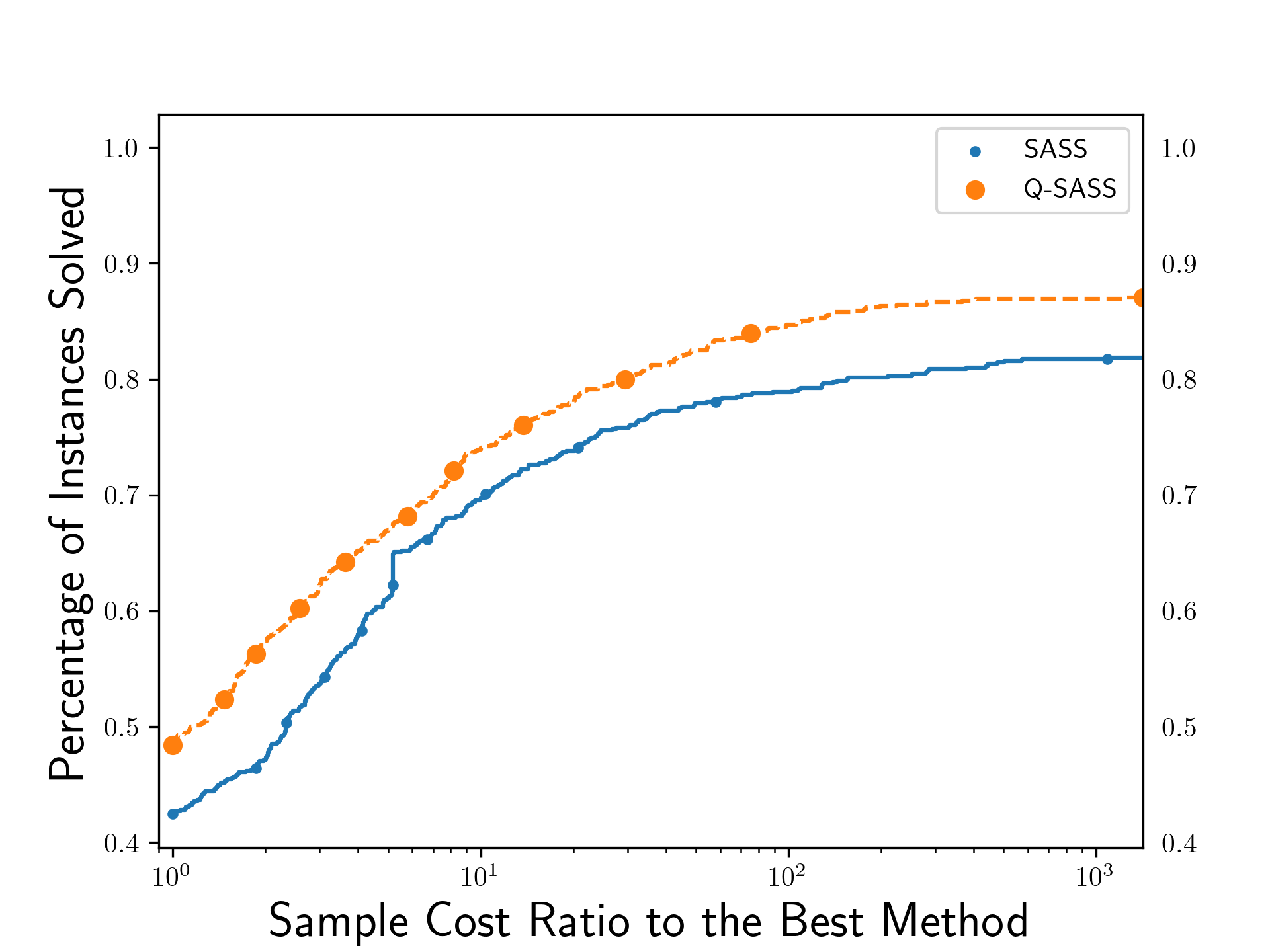}
\end{subfigure}
	\caption{Performance profiles for additive noise experiments \ml{on 810 problem instances (27 \texttt{CUTEst} problems, each with 30 random seeds).}
	The left plot uses the iteration count as the performance metric, while the right plot uses the observed sample count as the performance metric. }
	\label{fig:additive}
\end{figure}

\begin{figure}[h!]\centering
	\begin{subfigure}{.49\textwidth}
		\centering
		\includegraphics[trim=3 -20 20 0, clip, width=\linewidth]{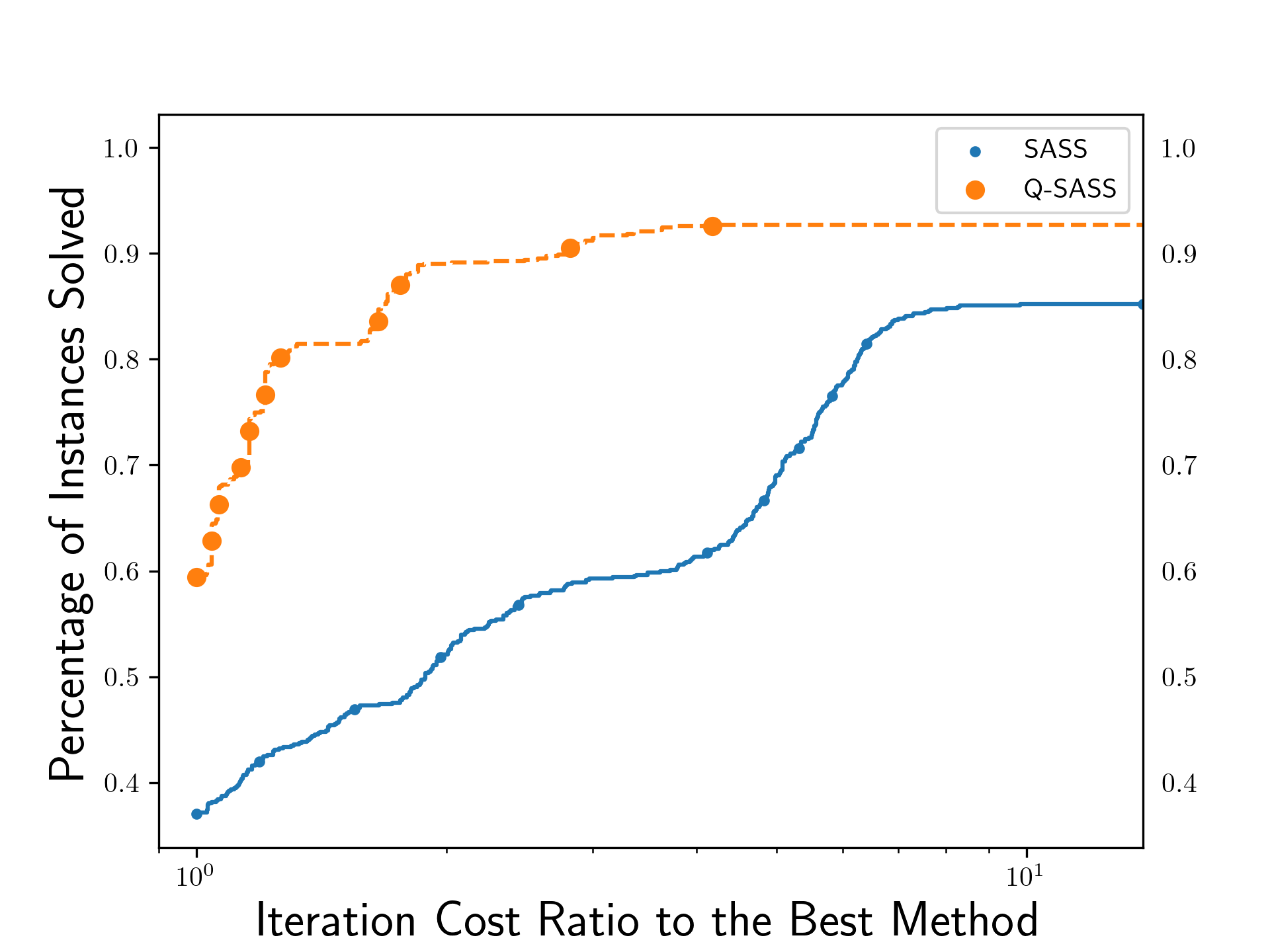}
	\end{subfigure}
	\hfill
	\begin{subfigure}{.49\textwidth}
		\centering
		\includegraphics[trim=3 -20 20 0, clip, width=\linewidth]{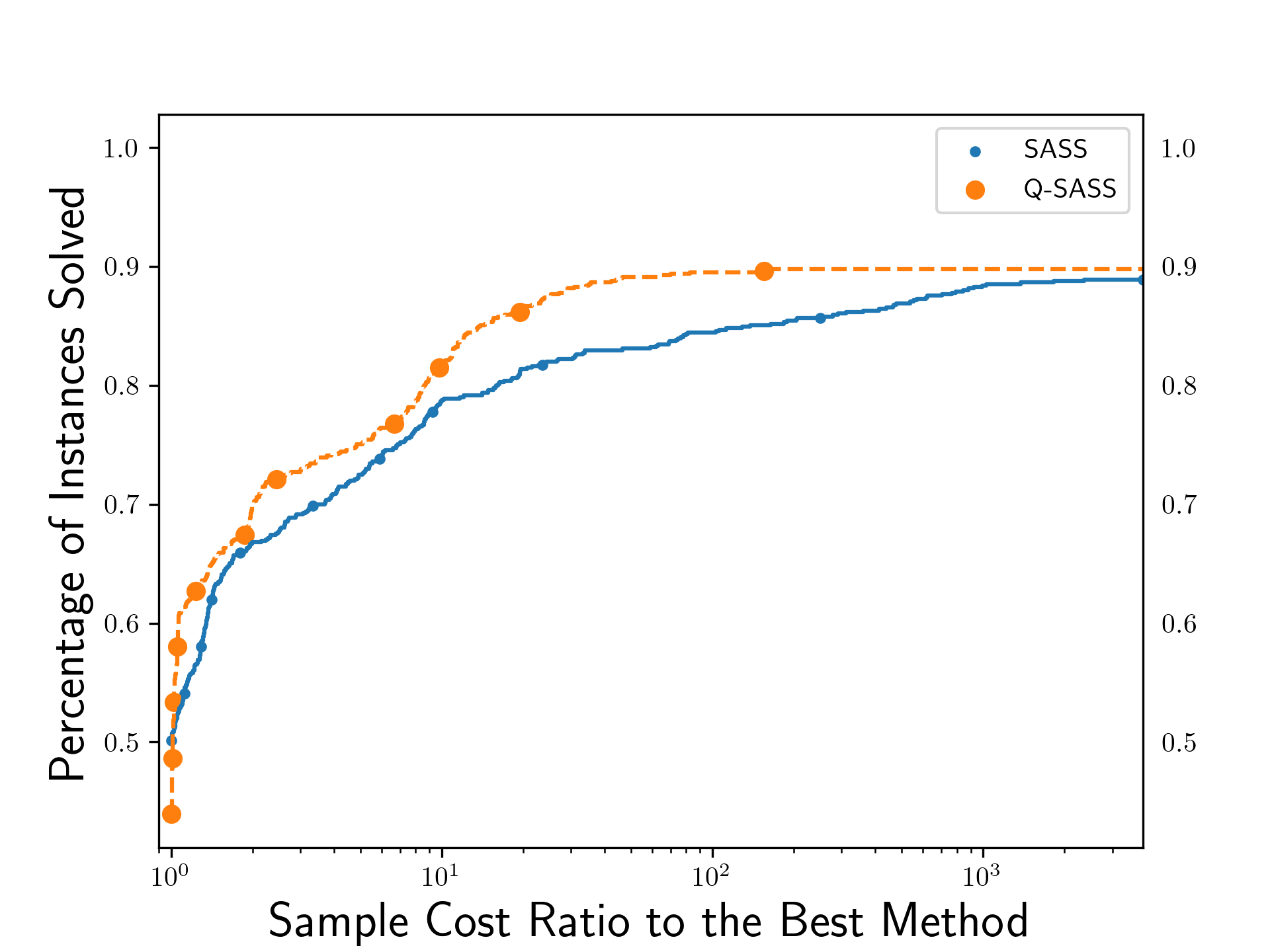}
	\end{subfigure}
	\caption{Performance profiles for multiplicative noise experiments\ml{ on 810 problem instances (27 \texttt{CUTEst} problems, each with 30 random seeds).}
	The left plot uses the iteration count as the performance metric, while the right plot uses the observed sample count as the performance metric. }
	\label{fig:multiplicative}
\end{figure}

\subsection{Advantage of bounding the eigenvalues of $B_k$}
As one may have noticed, there is no mechanism in \Cref{alg:enforce} that requires the memory parameter $M$ to be finite. 
In \Cref{fig:add_compare_bfgs} we compare Q-SASS from the previous set of experiments with a variant of Q-SASS that sets $M=\infty$, $\sigma_{\mathrm{ub}}^B=\infty$, and $\sigma_{\mathrm{lb}}^B=0$; in this case, \Cref{alg:enforce} never removes any information from the sets $S$ and $Y$, and the quasi-Newton updates are effectively just BFGS updates without bounding the eigenvalues. 
We refer to this variant of Q-SASS as Q-SASS-BFGS. 
We remark that the choices of $\sigma_{\mathrm{ub}}^B, \sigma_{\mathrm{lb}}^B$ in Q-SASS-BFGS are improper given the requirement of $\sigma_{\mathrm{lb}}^B>0$, and hence our convergence analysis in \Cref{sec:analysis} is inapplicable to Q-SASS-BFGS.
Nevertheless, in the same additive noise setting demonstrated in \Cref{fig:additive}, we see in \Cref{fig:add_compare_bfgs} that Q-SASS-BFGS is empirically a better method.

\begin{figure}[ht!]\centering
	\begin{subfigure}{.49\textwidth}
		\centering
		\includegraphics[trim=3 -20 20 0, clip, width=\linewidth]{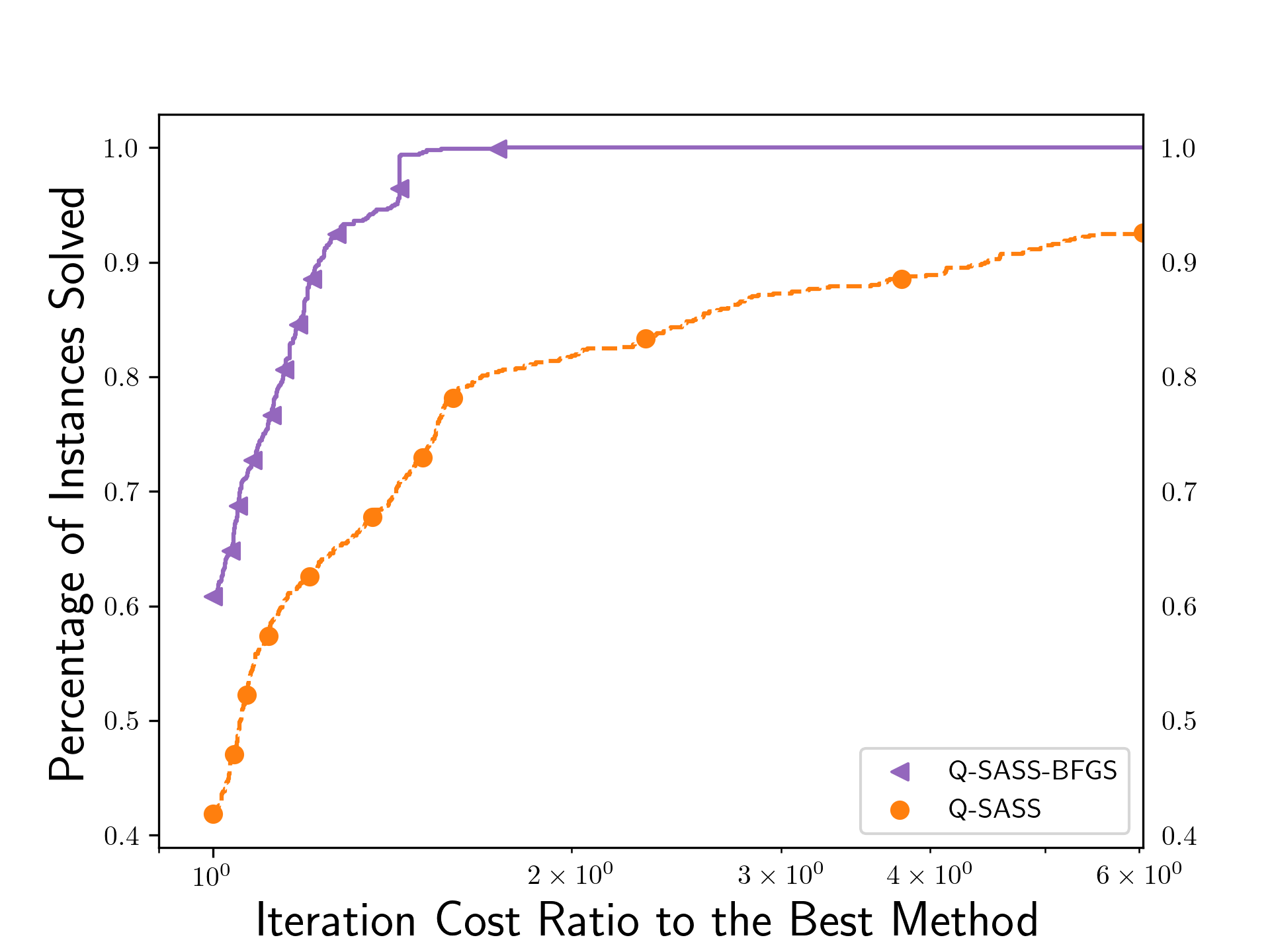}
	\end{subfigure}
	\hfill
	\begin{subfigure}{.49\textwidth}
		\centering
		\includegraphics[trim=3 -20 20 0, clip, width=\linewidth]{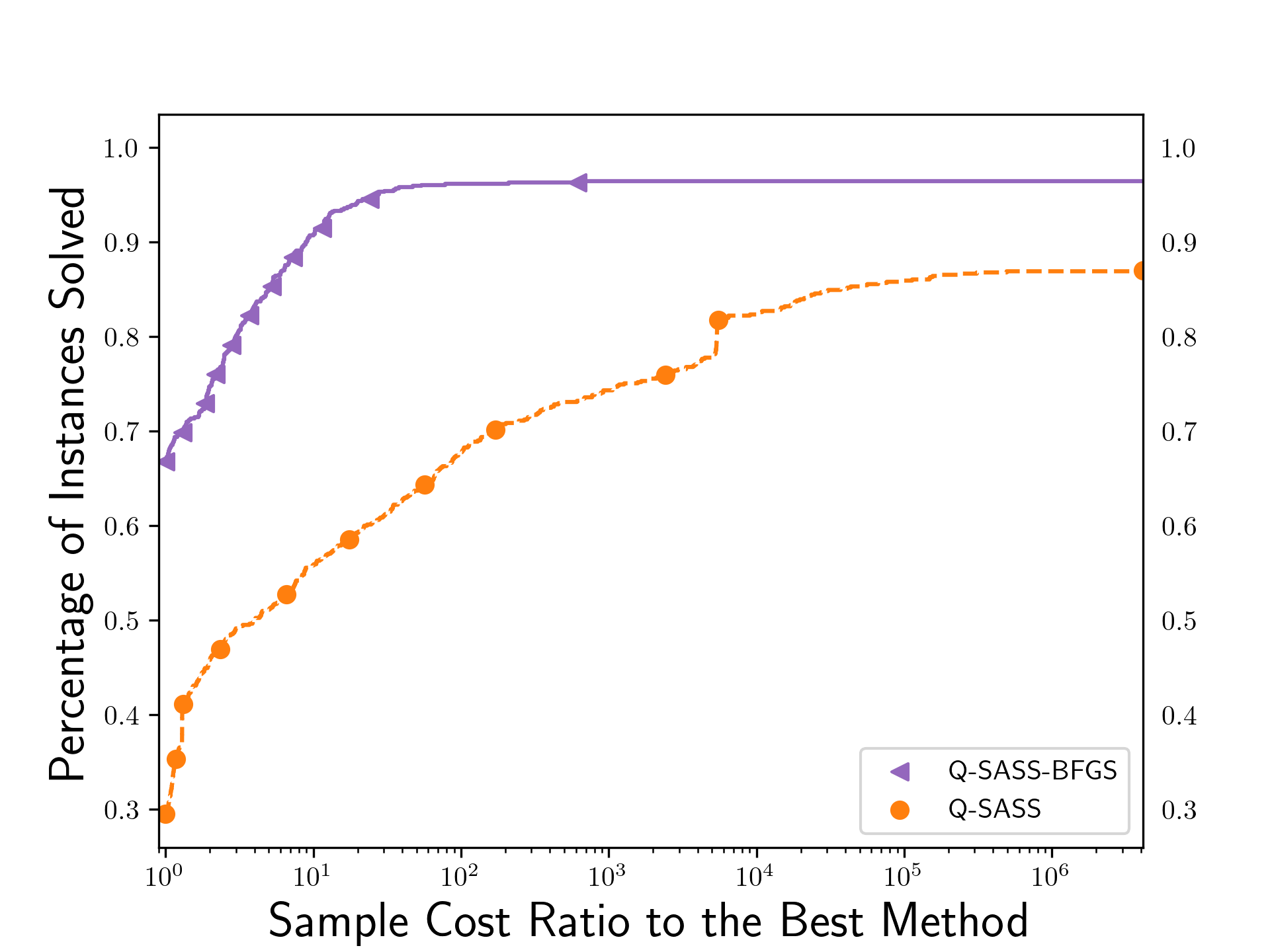}
	\end{subfigure}
        \caption{Performance profiles for additive noise experiments\ml{ on 810 problem instances (27 \texttt{CUTEst} problems, each with 30 random seeds).}
	The left plot uses the iteration count as the performance metric, while the right plot uses the observed sample count as the performance metric. }
	\label{fig:add_compare_bfgs}
	\begin{subfigure}{.49\textwidth}
		\centering
		\includegraphics[trim=3 -20 20 0, clip, width=\linewidth]{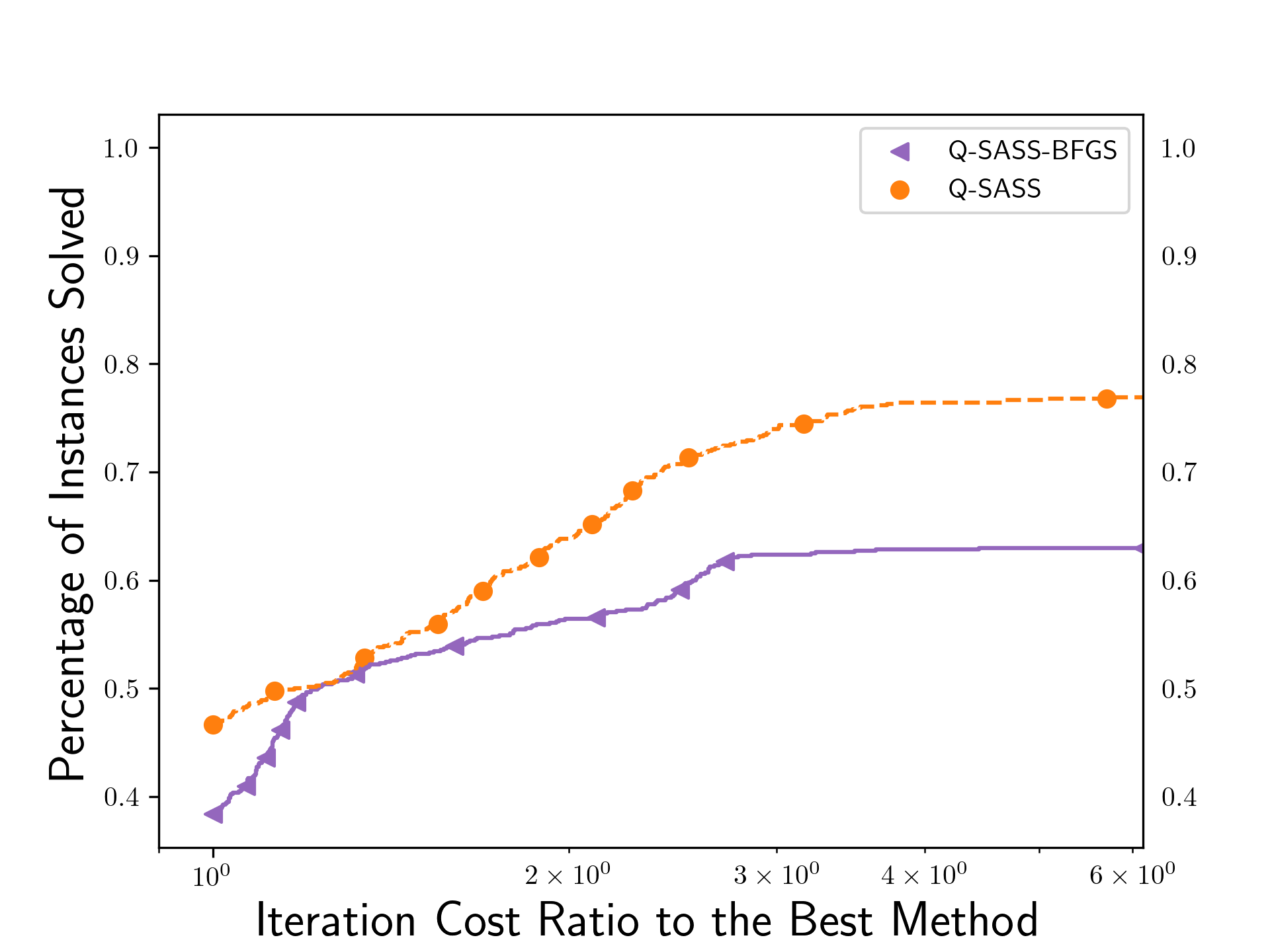}
	\end{subfigure}
	\hfill
	\begin{subfigure}{.49\textwidth}
		\centering
		\includegraphics[trim=3 -20 20 0, clip, width=\linewidth]{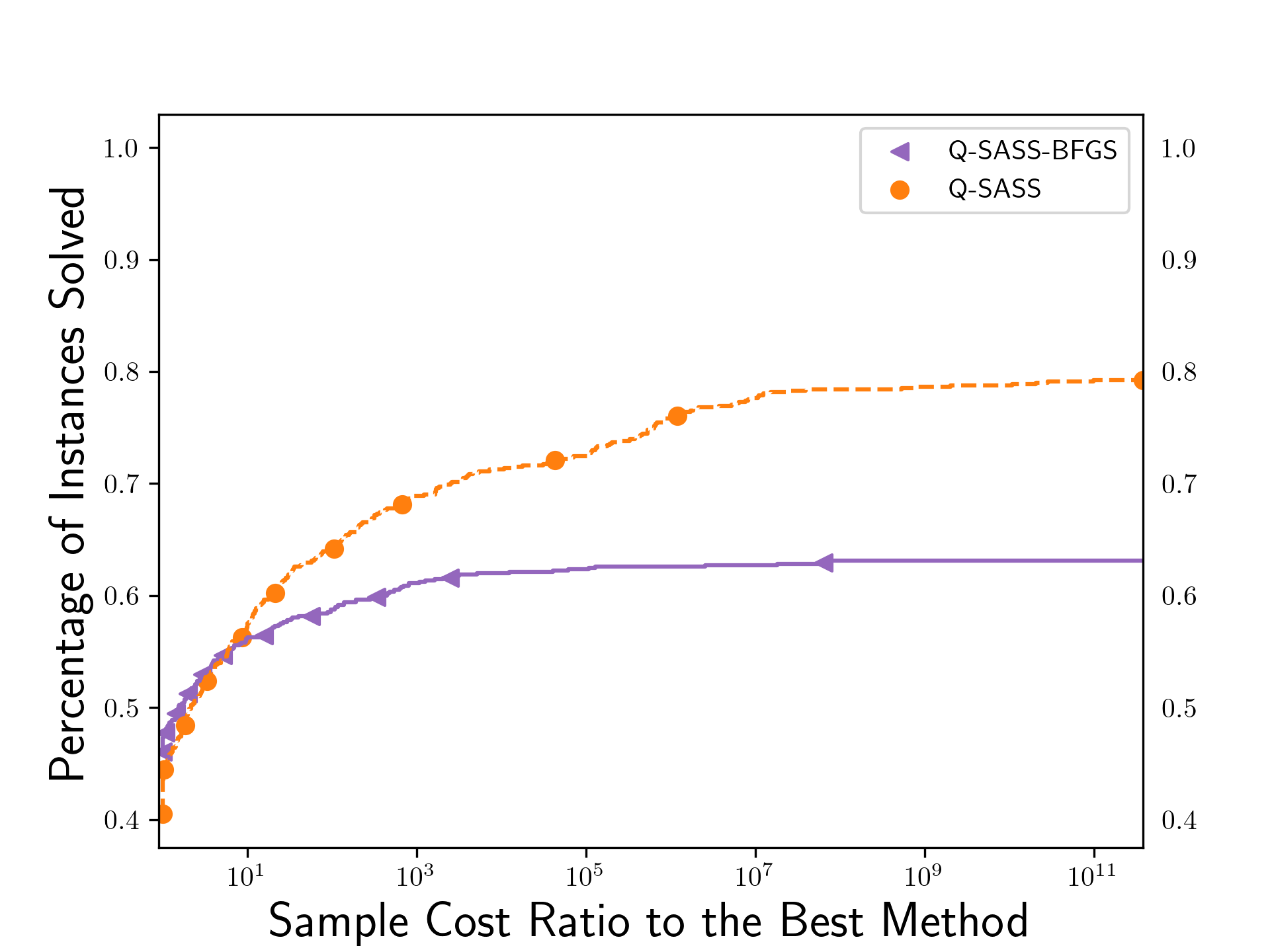}
	\end{subfigure}
	\caption{Performance profiles for mixed Gaussian noise experiments\ml{ on 810 problem instances (27 \texttt{CUTEst} problems, each with 30 random seeds).}
	The left plot uses the iteration count as the performance metric, while the right plot uses the observed sample count as the performance metric.  }
	\label{fig:mixed}
\end{figure}

One may hypothesize that the additive Gaussian noise setting is simply ``too well-behaved" in some sense, and so the gradient and the second-order estimates are never so poor as to completely ruin the quality of a search direction $d_k$. Thus, in such a well-behaved setting, there is no need to remove old curvature pairs. 
To compare with this setting, we tested a far more difficult noise setting, which we called the \emph{mixed Gaussian} setting. 
In this setting, we still generate estimates $f(x,\xi)$ of $\phi(x)$ according to \cref{eq:additive_noise}, 
but with the first-order oracle defined by
\begin{equation}
    g(x,\xi^\prime) = \begin{cases}
   \nabla\phi(x) + 10^{-6}\xi^\prime & \mbox{with probabilty } 0.8 \\
    \nabla\phi(x) + 10^{6}\xi^\prime & \mbox{with probabilty } 0.2.
    \end{cases}
    \label{eq:mixed_gaussian_noise}.
\end{equation}
The noise setting defined by \cref{eq:mixed_gaussian_noise} results in a reasonable gradient estimation in $80\%$ of the iterations but a terrible gradient estimate for the remaining $20\%$ of iterations.  
This is an exaggerated scenario modeling unreliable stochastic oracles, for instance, drift effects in a stochastic oracle, which frequently occur in current quantum computing \cite{proctor2020detecting}. 

In \Cref{fig:mixed}, we compare Q-SASS and Q-SASS-BFGS in the mixed Gaussian noise setting defined by \cref{eq:mixed_gaussian_noise}, with the stopping time for each \texttt{CUTEst} problem being $10^{-5}\|\nabla \phi(x_0)\|$, and the upper bound of the sample budget being $10^{23}$.
We observe that the performance of Q-SASS is more robust to the change in noise setting than is the performance of Q-SASS-BFGS.
By checking the number of times the eigenvalue bounds are violated in Q-SASS-BFGS under the additive noise and the mixed Gaussian noise setting, we see that under the additive setting, in general, only a very small percentage of the iterations would violate the bounds on the eigenvalues, while in the mixed Gaussian setting, a relatively large proportion of the iterations violate the eigenvalue bounds. For example,  for the COSINE problem, in the additive noise setting, every iteration of Q-SASS-BFGS satisfies the eigenvalue bounds, while in the mixed Gaussian setting, every run sees over $95\%$ of the iterations violating the eigenvalue bounds.

We conclude from these tests that a practical choice between Q-SASS and Q-SASS-BFGS should depend on the specific noise setting, but in general, Q-SASS may be a more robust choice when noise is particularly adverse since bounding the eigenvalues of $B_k$ has an effect of denoising $d_k$ and can make $d_k$ more robust against the noise of the gradient estimation $g_k$. Note that the choices of the eigenvalue bounds are hyper-parameters one can choose based on the specifics of the problem; at lower noise levels, more relaxed bounds may be chosen, whereas at higher noise levels, selecting a more stringent bound is beneficial. 


\subsection{Effects of parameters $\kappa$ and $\mu$}
We remind the reader that in \Cref{sec:synthetic} we defined two parameters $\kappa$ and $\mu$, which directly control the accuracy of the zeroth- and first-order oracles
with the defaults in Q-SASS set as $1$ and $0.01$, respectively. 
We observed in initial experiments that the effect of varying $\mu$ was relatively inconsequential compared with the effect of varying $\kappa$, so we chose to show several values of $\kappa$ for a fixed value of $\mu$ in \Cref{fig:search}. 
\Cref{fig:search} shows the performance of Q-SASS in the additive noise setting while varying the $\kappa$ values.
We again use performance profiles but augment them with data profiles \cite{more2009benchmarking} to illustrate more fully the difference it makes in the Q-SASS solver. 

From the definition of $\kappa$ (recall \cref{eq:first_order}), smaller values of $\kappa$ imply more accurate gradient estimates.
This amounts to a trade-off in that small values of $\kappa$ should yield a stochastic process that more closely tracks a deterministic QN method, thereby likely yielding lower iteration counts, but at the expense of a greater sample count in order to achieve the accuracy required by \cref{eq:first_order}.  
Our takeaway from the experiments summarized in \Cref{fig:search} is that the parameter $\kappa$ should be viewed as our best practical control for managing this trade-off depending on a user's needs.
This parameter is especially relevant to the trade-off in latency time versus shot acquisition in hybrid quantum-classical computing that was studied in \cite{menickelly2022latency}. 

\begin{figure}[h!]\centering
	\begin{subfigure}{.49\textwidth}
		\centering
		\includegraphics[trim=3 0 20 0, clip, width=\linewidth]{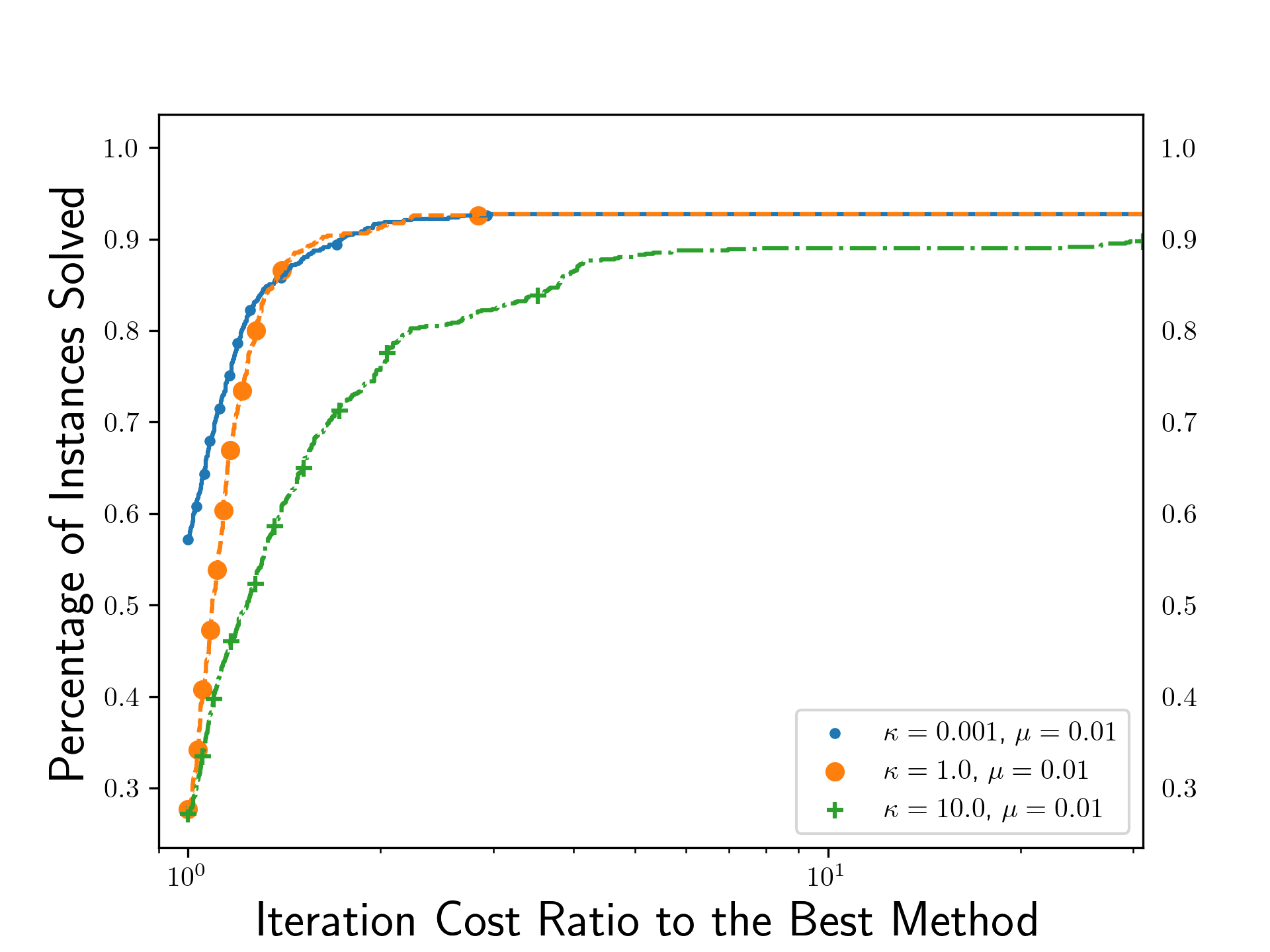}
	\end{subfigure}
	\hfill
	\begin{subfigure}{.49\textwidth}
		\centering
		\includegraphics[trim=3 0 20 0, clip, width=\linewidth]{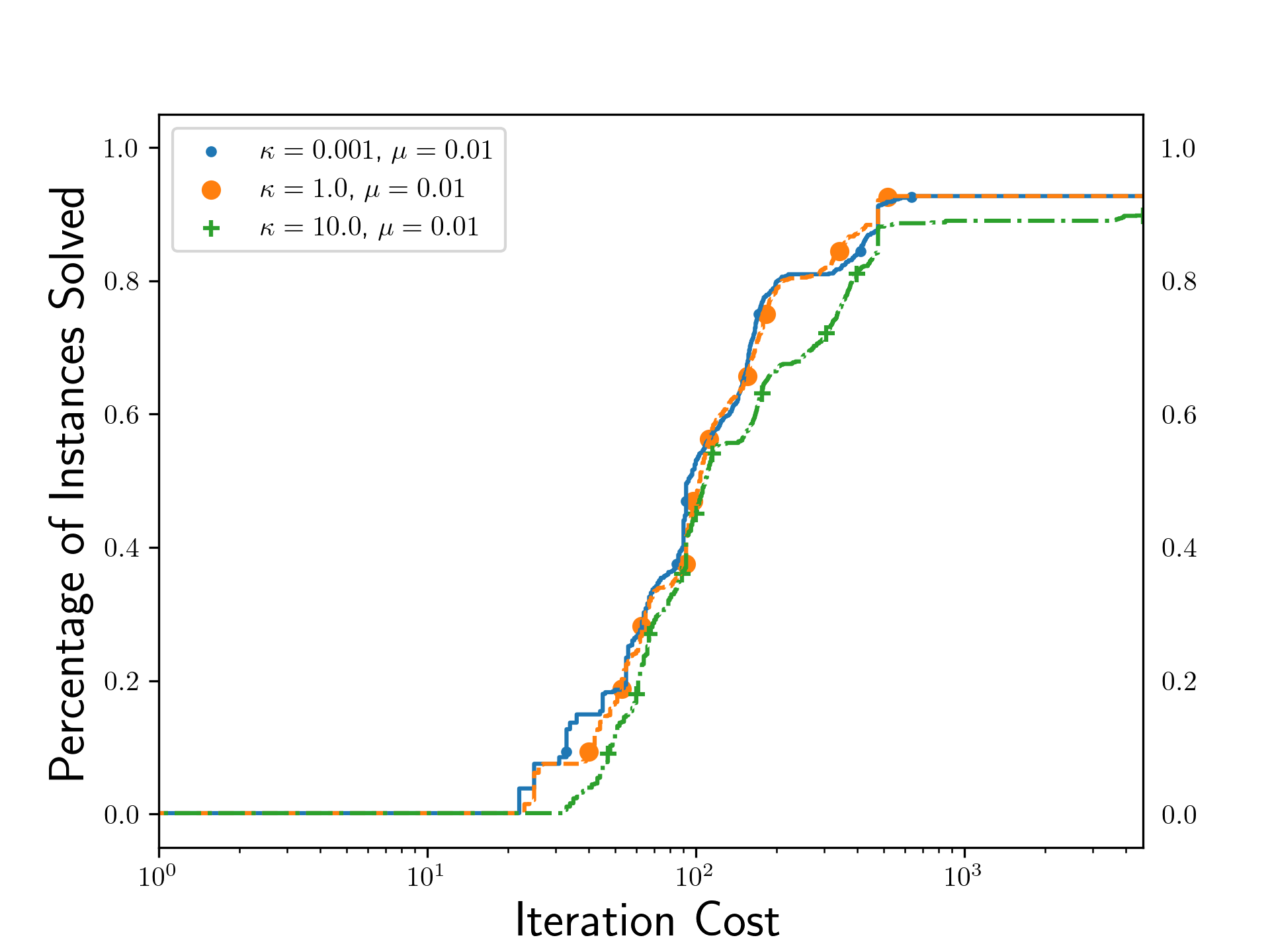}
	\end{subfigure}
	\begin{subfigure}{.49\textwidth}
		\centering
		\includegraphics[trim=3 -20 20 0, clip, width=\linewidth]{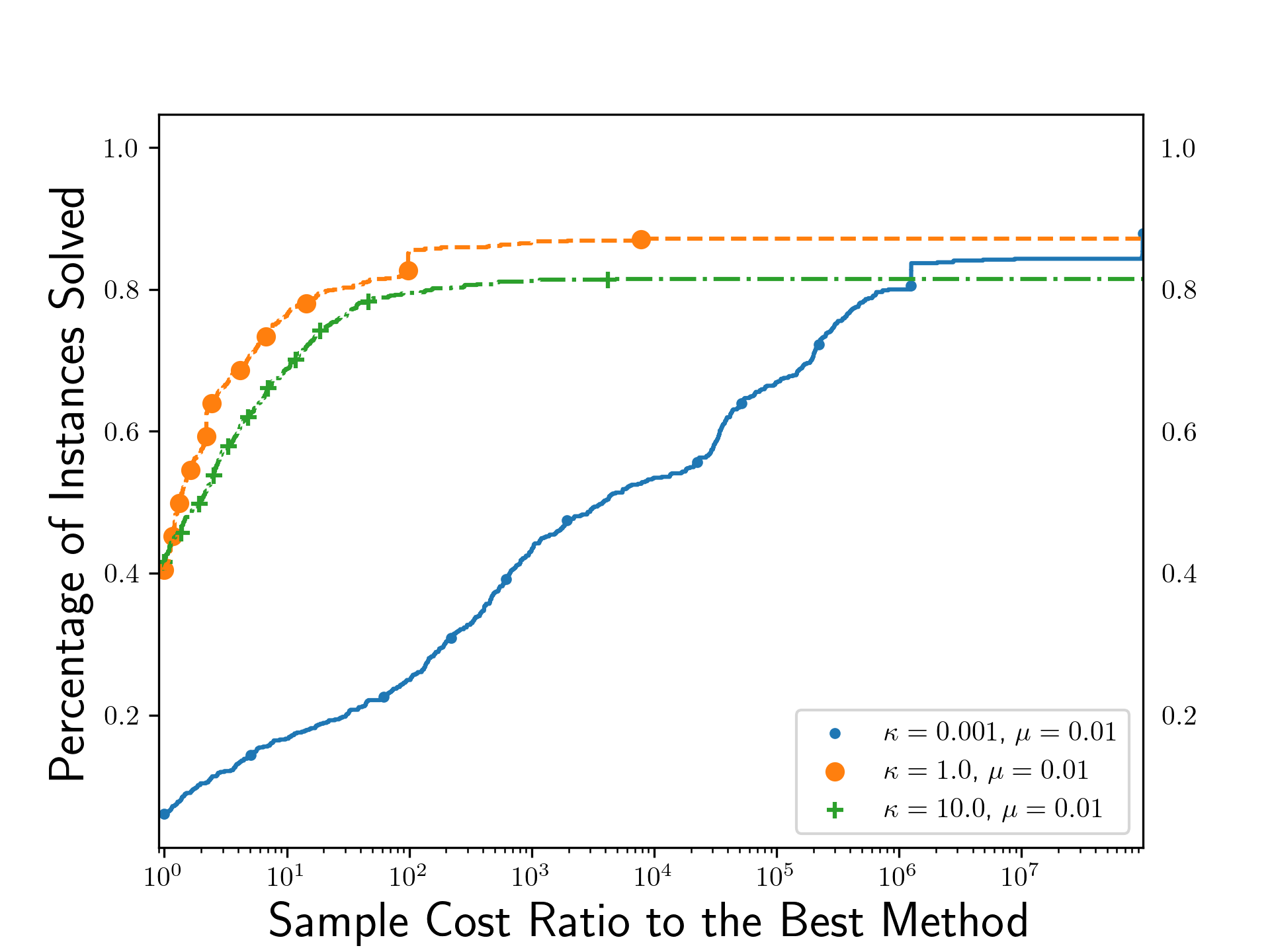}
	\end{subfigure}
	\begin{subfigure}{.49\textwidth}
		\centering
		\includegraphics[trim=3 -20 20 0, clip, width=\linewidth]{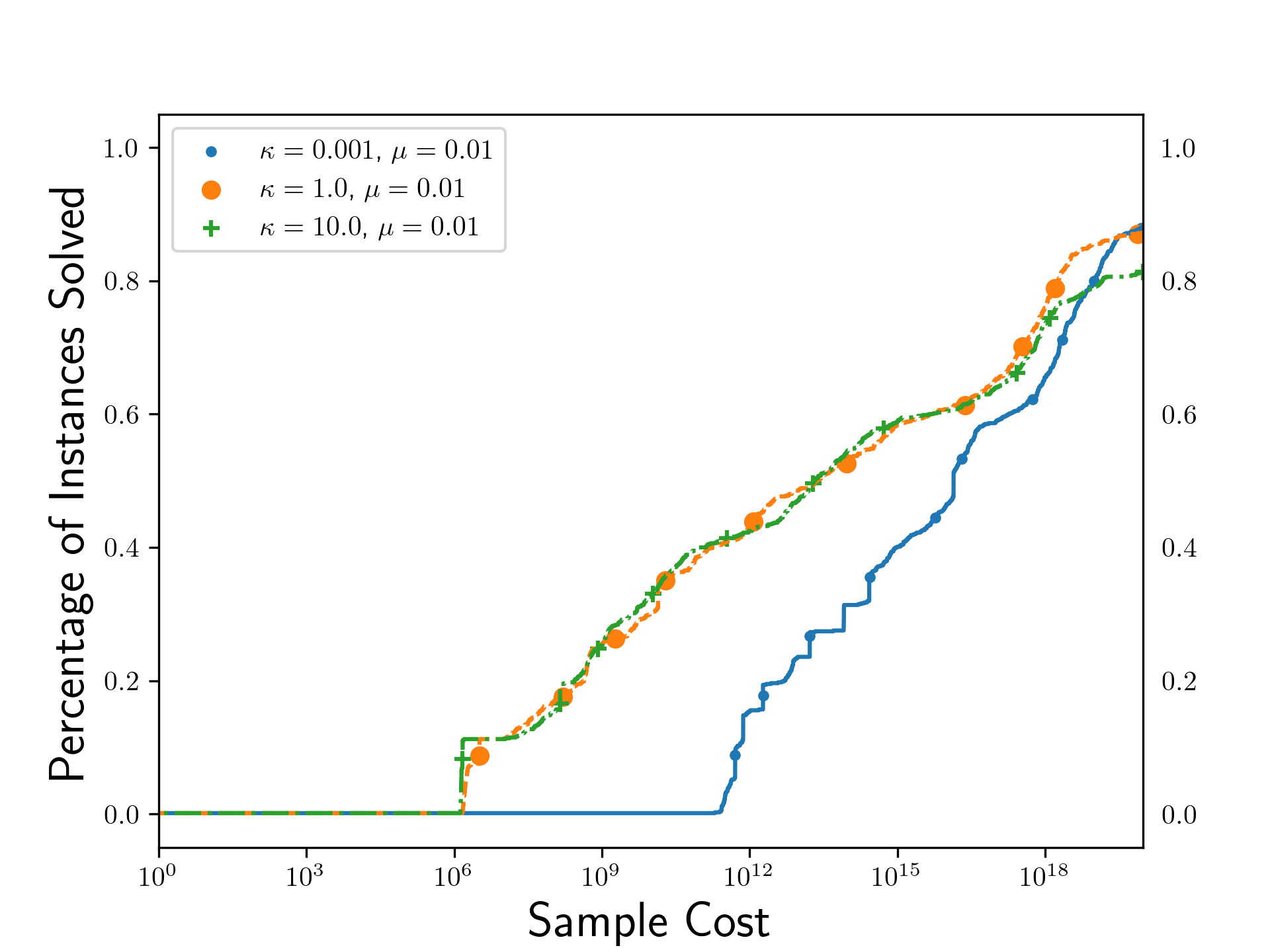}
	\end{subfigure}
	\caption{Profiles in the additive noise setting\ml{ on 810 problem instances (27 \texttt{CUTEst} problems, each with 30 random seeds) }using different values of $\kappa$. The left two plots are performance profiles, and the right two plots are data performance profiles. The top two plots use iteration count as the metric, and the bottom two use sample cost as the metric. }
	\label{fig:search}
\end{figure}

\subsection{VQE problems} \label{vqe_sec}
Given our previously discussed motivations in quantum computing, we employ the same VQE test problems as studied in \cite{menickelly2022latency}.
We defer detailed discussion of these optimization problems to that paper, but we note that we are using the H$_2$ and LiH problems from \cite{menickelly2022latency} with ansatz depths that result in $n=3$- and $n=16$-dimensional optimization problems, respectively. 
Motivated by the aforementioned issue of chemical accuracy, we define our stopping times as reaching a point that has a function value within $10^{-3}$ of the true minimum, and we set $\varepsilon_f = 0.0001$ and $\kappa =0.5$.

The zeroth-order oracle is as described in \Cref{sec:motivation}.
The first-order oracle is obtained by substituting the exact function values in \cref{eq:magic_dd} with independent outputs from the zeroth-order oracle. 
The per-iteration sample budgets for the zeroth- and first-order oracles still follow \cref{eq:nfk_ngk}. 
As seen in \cref{eq:magic_dd}, gradient approximations are necessarily performed coordinate-wise, and so we may use coordinate-wise variance estimates of the partial derivatives to distribute the gradient sample budget $N_{g,k}$ optimally in order to minimize the variance of the full gradient estimate $\mathbb{V}(g\left(x_k,\Xi^{\prime}\right))$. 
In particular, we solve 
\begin{equation}\label{opt_prog1}
	\begin{array}{ll}
		\mbox{minimize}_w \quad & \sum_{i=1}^n  \frac{\var_i}{w_i} \\
		\mbox{subject to} & \sum_{i=1}^n w_i=1 \\
		& w\geq 0,
	\end{array}
\end{equation}
where $\var_i$ denotes the sample variance of the $i$th partial derivative estimate
and
the decision variables $w_i$ represent the percentages of the gradient sample budget $(N_{g,k})$ allocated to the $i$th partial derivative estimate.
One can easily verify that the solution to the convex program
\cref{opt_prog1} suggests that the optimal way to distribute the sample budget $N_{g,k}$ is to distribute samples proportionally to the standard deviations of the partial derivatives.

For each optimization problem in the VQE setting, we generate $30$ instances by fixing a unique random seed at the beginning of the optimization run. 
For these tests,
the maximum number of allowable iterations is $500$, and the total sample budget for each run is set to $10^{10}$.
The results,  shown in performance profiles in \Cref{fig:quantum}, 
clearly demonstrate a preference for using Q-SASS over SASS in terms of both iteration and sample costs. 

\begin{figure}[h!]\centering
	\begin{subfigure}{.49\textwidth}
		\centering
		\includegraphics[trim=3 -20 20 0, clip, width=\linewidth]{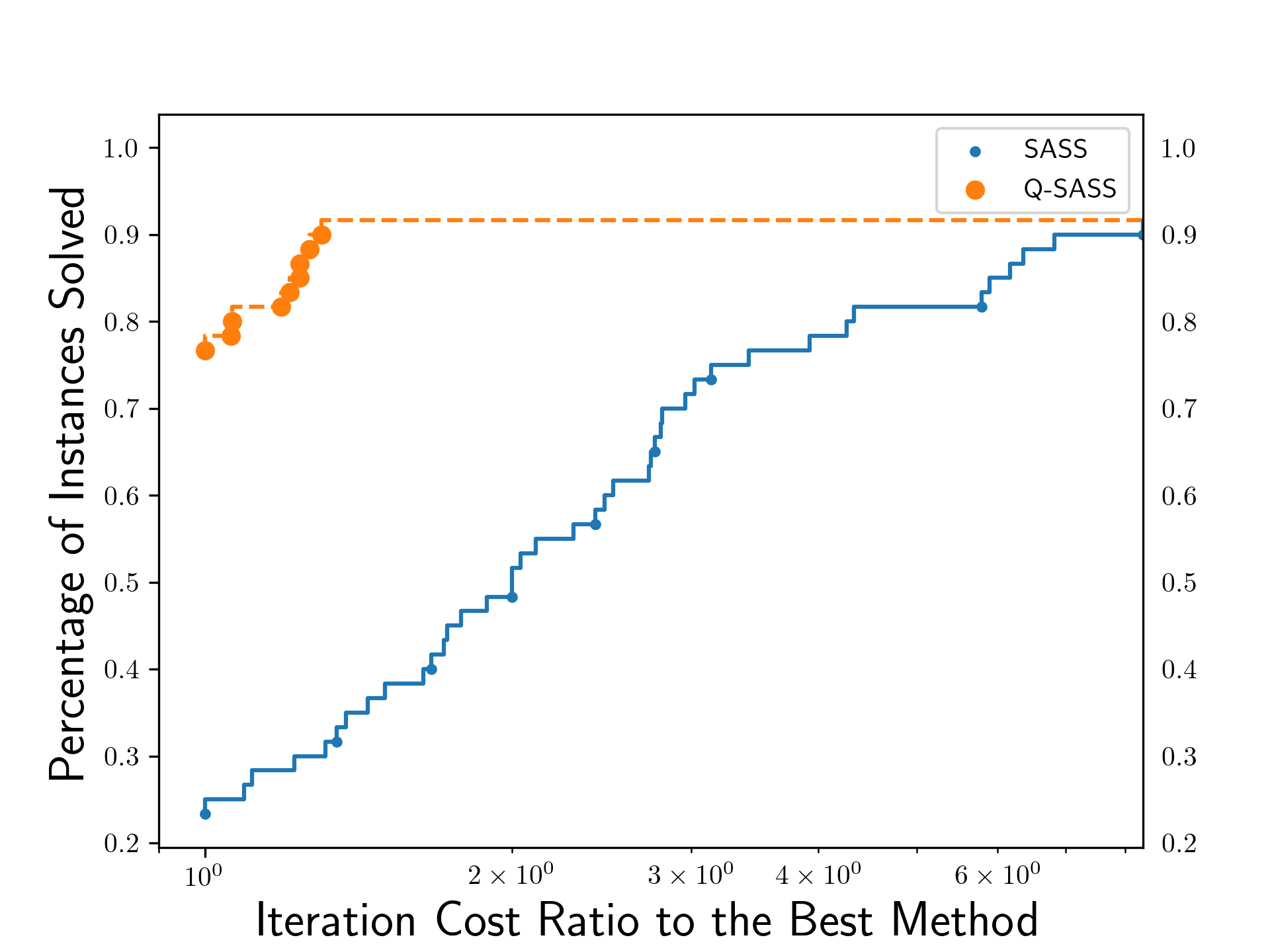}
	\end{subfigure}
	\hfill
	\begin{subfigure}{.49\textwidth}
		\centering
		\includegraphics[trim=3 -20 20 0, clip, width=\linewidth]{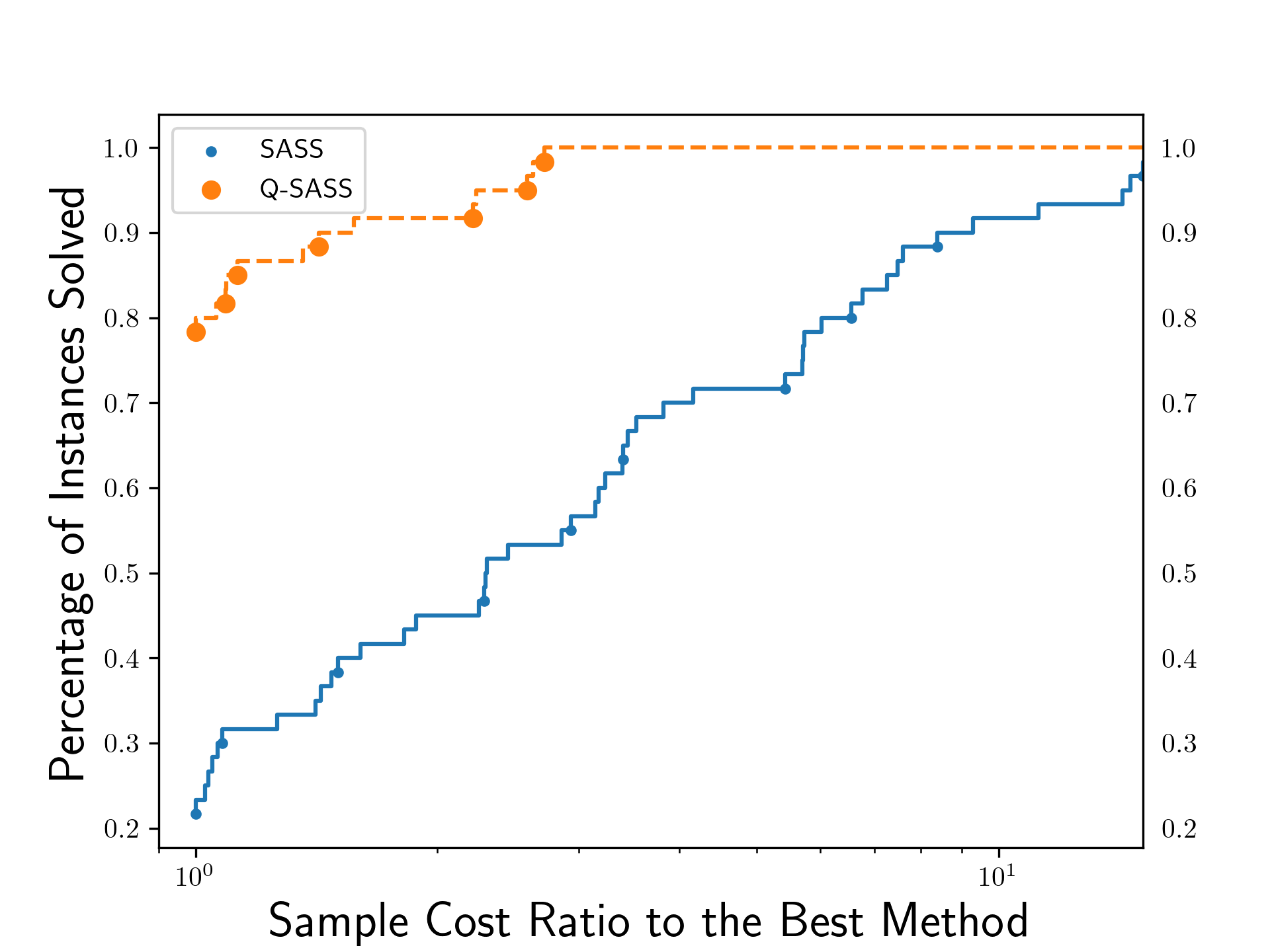}
	\end{subfigure}
	\caption{Performance profiles for VQE problems.
	The left plot uses the iteration count as the performance metric, while the right plot uses the observed sample count as the performance metric.}
	\label{fig:quantum}
\end{figure}

\subsection{Finite-difference derivative-free optimization}

We also experiment with Q-SASS in the derivative-free setting on the same \texttt{CUTEst} problems as in \Cref{sec:synthetic}; that is, we assume that no first-order oracle is directly available and must be estimated via repeated queries of the zeroth-order oracle.
In particular, the zeroth-order oracle will still be defined as it was in \cref{eq:additive_noise} for additive noise, but the first-order oracle will now only be approximated via a forward finite-difference gradient estimate. 
 Specifically, given a difference parameter $h>0$, the $i$th partial derivative is estimated via
\begin{equation}
\label{eq:finitediff}
g(x;\xi^\prime)_i=  \frac{f(x+h e_i;\xi^\prime_i)-f(x;\xi^\prime_0)}{h},
\end{equation}
where $\xi^\prime = \{\xi^\prime_0,\xi^\prime_1, \cdots, \xi^\prime_n\}\in \R\times \cdots \times\R$ 
and $e_i$ is the $i$th elementary basis vector.
Let $n$ be the dimension of the problem, and let $N_{g,k}$ be the total sample budget for the gradient estimation. Following similar reasoning as in \Cref{vqe_sec}, we distribute the gradient sample budget $N_{g,k}$ at iteration $k$ by the following logic.
We first allocate $s_0=\frac{N_{g,k}}{n+1}$ many samples to estimate the function value at $x_k$ 
Then, for each $i=1,\dots,n$ in \cref{eq:finitediff} the number of samples $s_i$ used to estimate $\phi(x+h e_i)$ is allocated again according to the estimated standard deviation of the $i$th partial derivative, using the remaining sample budget. Specifically, $s_i=w_i\frac{nN_{g,k}}{n+1}$, with the $w_i$'s being the optimal solution of \eqref{opt_prog1}.

Inspired by work in \cite{berahas2022theoretical} and \cite{more2012estimating},
we choose the difference parameter $h$  according to the formula
$$h = 2\sqrt{\frac{e_{\mathrm{std}}}{\bar L}} + 10^{-8},$$
where $e_{\mathrm{std}}$ is an estimate of $\left(\frac{\var f(x,\Xi^\prime_0)}{s_0}\right)^{1/2}$ (i.e., an estimate of the standard deviation of the function value estimation) and $\bar L$ is an estimate of the gradient Lipschitz constant. 
In the experiment, we set $\bar L = \|\nabla^2 \phi(x_0)\|$.
The purpose of the term $10^{-8}$ is to guarantee that $h$ is at least the square root of the typical roundoff error from double-precision floating-point arithmetic. 
Using \cref{eq:nfk_ngk}, plugging in the value of $h$ into the gradient formula, and then rearranging the terms, we obtain the number of samples needed for the gradient estimation at iteration $k$:
$$N_{g,k}
\geq 
\frac{L^2\left(\sum_{i=1}^n\var f(x,\Xi^\prime_0)+\var f(x+he_i,\Xi^\prime_i)\right)^2}{16(n+1)\,\var f(x,\Xi^\prime_0)\cdot\delta^2\varepsilon_{g,k}^4 }.$$
Since all the function estimations have the same constant variance in the additive noise setting, the above bound for gradient sample size can be simplified to
$$N_{g,k}
\geq 
\frac{L^2n^2\,\var f(x,\Xi^\prime_0)}
{4(n+1)\cdot\delta^2\varepsilon_{g,k}^4 }.$$
Clearly, in this setting, the optimal way to distribute the gradient sample budget is to distribute it evenly across the coordinates.

We used the same parameter settings as in the experiments in \Cref{sec:synthetic}, with the total iteration budget being $300$ and the total sample budget being $10^{23}$. 
The results of these experiments are shown in \Cref{fig:FD}. 
We again observe a preference, in terms of iterations and samples, for Q-SASS over SASS in this setting. 

\begin{figure}[h!]\centering
	\begin{subfigure}{.49\textwidth}
		\centering
		\includegraphics[trim=3 -20 20 0, clip, width=\linewidth]{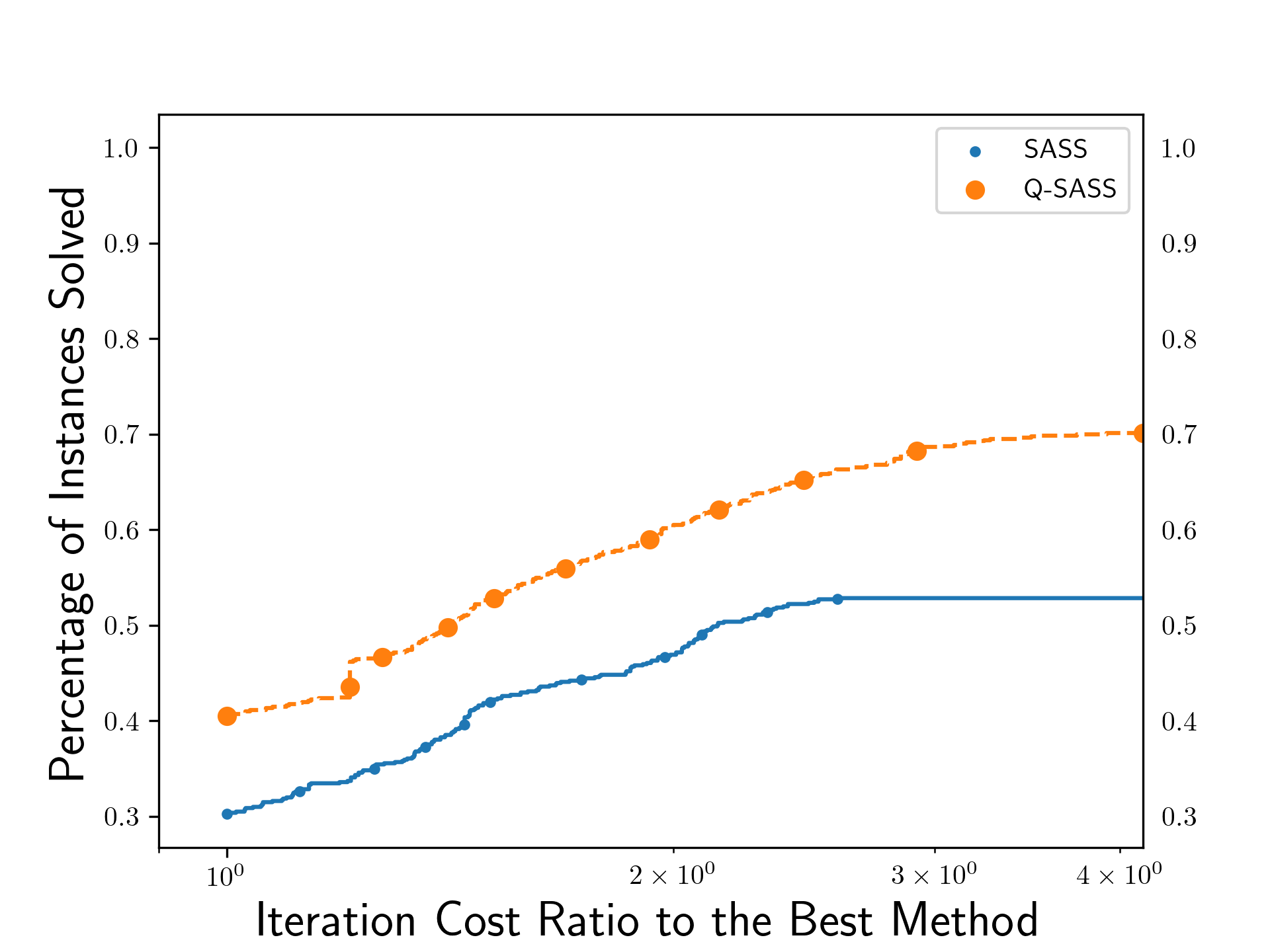}
	\end{subfigure}
	\hfill
	\begin{subfigure}{.49\textwidth}
		\centering
		\includegraphics[trim=3 -20 20 0, clip, width=\linewidth]{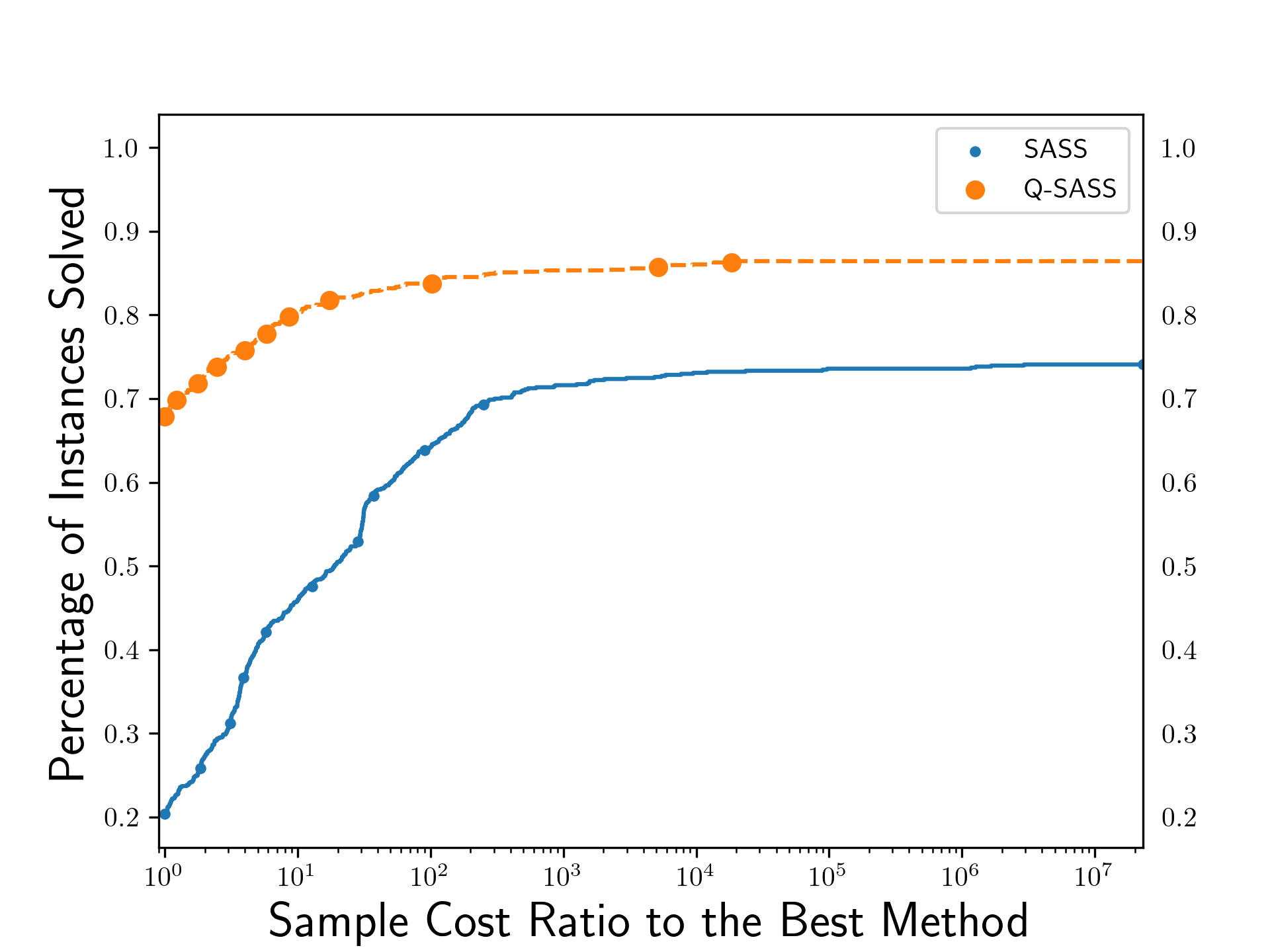}
	\end{subfigure}
	\caption{Performance profiles for finite difference experiments\ml{ on 810 problem instances (27 \texttt{CUTEst} problems, each with 30 random seeds).}
	The left plot uses the iteration count as the performance metric, while the right plot uses the observed sample count as the performance metric.}
	\label{fig:FD}
\end{figure}

\section{Conclusion} \label{sec: conclusion}
In this paper we designed, analyzed, and tested a stochastic quasi-Newton method without the advantage of exploiting common random numbers. 
Function information is  available only through possibly biased probabilistic zeroth- and first-order oracles, which can have unbounded errors. 
For nonconvex functions, we demonstrated that the algorithm achieves an $\varepsilon$-stationary point in $O(\frac{1}{\varepsilon^2})$ iterations with overwhelmingly high probability.
For strongly convex functions, the algorithm converges linearly to an $\varepsilon$-optimal point with overwhelmingly high probability. 
The advantage of the algorithm is confirmed by the empirical experiments.

\subsection*{Acknowledgments}
	This material was based upon work supported by the U.S.\ Department of
	Energy, Office of Science, Office of Advanced Scientific Computing
	Research, applied mathematics and ARQC programs under Contract Nos.\
    DE-AC02-05CH11231 and DE-AC02-06CH11357. Miaolan Xie was partially  supported by NSF TRIPODS grant 17-40796.
\bibliographystyle{siamplain}
\bibliography{refs}

\appendix

\section{Proof of \Cref{prop:ass_nonconvex}}
\label{sec:propncproof}
\begin{proof}
The proof of this result follows the same reasoning as the proof of Proposition 1 in \cite{jin2021high}. 
In fact, the proofs of the first two claims we need to prove in \Cref{ass:alg_behave} are the same as in that work, and we refer the reader to \cite{jin2021high} directly. 
We continue by proving the remaining three claims for our setting. 
 
\begin{itemize}		
		\item[3.] 
		Since iteration $k$ is true, we have that $\norm{G_k - \nabla \phi(X_k)} \leq \max\{\varepsilon_g,  \min\{\tau, \kappa A_k\} \norm{G_k}\}$.
		We consider two cases:
		\begin{enumerate}
			\item Suppose $\norm{G_k - \nabla \phi(X_k)} \leq \min\{\tau, \kappa A_k\} \norm{G_k}$. 
			By the triangle inequality,  
			$$\norm{G_k} \geq \frac{1}{1+\min\{\tau, \kappa A_k\} } \norm{\nabla\phi(X_k)}\geq \frac{1}{1+\tau} \norm{\nabla\phi(X_k)}.$$ 
			Together with the fact that iteration $k$ is successful and \Cref{prop:innerproduct}, we obtain 
			$$f(X_{k+1}) - f(X_k) \leq -  A_k\theta\dotp{D_k,G_k} + 2\varepsilon_f \leq  -  A_k\theta\sigma_{\mathrm{lb}}\norm{G_k}^2 + 2\varepsilon_f
			\leq- \frac{ A_k\theta\sigma_{\mathrm{lb}}\norm{\nabla\phi(X_k)}^2}{(1+\tau )^2} + 2\varepsilon_f.$$
			\item Suppose $\norm{G_k - \nabla \phi(X_k)} \leq \varepsilon_g$. Since $k < T_\varepsilon$, we have $\norm{\nabla \phi(X_k)} > \varepsilon\geq \frac{\varepsilon_g}{\eta}$. This implies that $\norm{G_k - \nabla \phi(X_k)} \leq \eta\norm{\nabla\phi(X_k)}$. 
			Rearranging via the triangle inequality, we have
			$\norm{G_k} \geq (1-\eta)\norm{\nabla\phi(X_k)}.$
			Combined with the fact that iteration $k$ is successful, 
			\begin{equation*}
			\begin{array}{ll}
			f(X_{k+1}) - f(X_k) & \leq -  A_k\theta\dotp{D_k,G_k} + 2\varepsilon_f
			\leq 
			-  A_k\theta\sigma_{\mathrm{lb}}\norm{G_k}^2 + 2\varepsilon_f\\
			 & \leq  -  A_k\theta\sigma_{\mathrm{lb}}(1-\eta)^2\norm{\nabla\phi(X_k)}^2 + 2\varepsilon_f.
			\end{array}
			\end{equation*}
		\end{enumerate}
		Combining the above two cases, we have that on any true, successful iteration $k < T_\varepsilon$, 
		$$f(X_{k+1}) - f(X_k) \leq -\min\left\{\frac{1}{(1+\tau )^2}, \, (1-\eta)^2\right\} A_k\theta\sigma_{\mathrm{lb}}\norm{\nabla\phi(X_k)}^2 + 2\varepsilon_f.$$
		Because $k < T_\varepsilon$, we have that $\norm{\nabla\phi(X_k)} > \varepsilon$, and so
		$$f(X_{k+1}) - f(X_k) \leq -h( A_k) + 2\varepsilon_f.$$ 
		Finally,  because $E_k+E_k^+\leq 2\varepsilon_f$ on true iterations, we have that
		$$\phi(X_{k+1}) - \phi(X_k) \leq -h( A_k) + 4\varepsilon_f.$$
		Recalling that $Z_k = \phi(X_k) - \phi^*$, we conclude that $Z_{k+1} - Z_k = \phi(X_{k+1}) - \phi(X_k)$.
		This completes the proof. 
		
		\item[4.] We first show that if $ A_k \leq \bar{\alpha}$ and $I_k = 1$, then 
		\begin{equation}
			\label{eq:ass1iinoiseless}
			\phi(X_k -  A_kD_k) \leq \phi(X_k) -  A_k\theta\dotp{D_k,G_k}.
		\end{equation} 
		Since $I_k=1$ (i.e., the $k$th iteration is true),  
		$\norm{G_k - \nabla \phi(X_k)} \leq \max\{ \min\{\tau, \kappa A_k\} \norm{G_k}, \varepsilon_g\}$. 
		Similar to the proof of the third claim, we consider two cases:
		\begin{enumerate}
			\item Suppose $\norm{G_k - \nabla \phi(X_k)} \leq  \min\{\tau, \kappa A_k\} \norm{G_k}$. 
			By \Cref{ass:Lipschitz},
			$$
			\phi\left(X_{k+1}\right) \leq \phi\left(X_{k}\right)+ A_{k} D_{k}^{T} \nabla \phi\left(X_{k}\right)+\frac{L}{2}\left\| A_{k} D_{k}\right\|^{2} .
			$$
			Thus by the Cauchy--Schwarz inequality and \Cref{prop:innerproduct}, we obtain for every true iteration
			$$
			\begin{aligned}
				\phi\left(X_{k}- A_{k} D_{k}\right) & \leq \phi\left(X_{k}\right)- A_{k} D_{k}^{T} \nabla \phi\left(X_{k}\right)+\frac{ A_{k}^{2} L}{2}\left\|D_{k}\right\|^{2} \\
				&=\phi\left(X_{k}\right)- A_{k} D_{k}^{T}\left(\nabla \phi\left(X_{k}\right)-G_{k}\right)- A_{k} D_{k}^{T} G_{k}+\frac{ A_{k}^{2} L}{2}\left\|D_{k}\right\|^{2} \\
				&\leq \phi\left(X_{k}\right)+ A_{k}\left\|D_{k}\right\|\left\|\nabla \phi\left(X_{k}\right)-G_{k}\right\|- A_{k} D_{k}^{T} G_{k}+\frac{ A_{k}^{2} L}{2}\left\|D_{k}\right\|^{2} \\
				& \leq \phi\left(X_{k}\right)+{\kappa A_{k}^2 }\left\|D_{k}\right\|\left\|G_{k}\right\|- A_{k} D_{k}^{T} G_{k}+\frac{ A_{k}^{2} L \sigma_{\mathrm{ub}}}{2}\left\|D_{k}\right\|\left\|G_{k}\right\| \\
				& \leq \phi\left(X_{k}\right)- A_{k} D_{k}^{T} G_{k}+\left(\kappa A_{k}^2+\frac{ A_{k}^{2} L \sigma_{\mathrm{ub}}}{2}\right)\left\|D_{k}\right\|\left\|G_{k}\right\|.
			\end{aligned}
			$$
			Note that \cref{eq:ass1iinoiseless} holds whenever
			$$
			\begin{aligned}
				&\phi\left(X_{k}\right)- A_{k} D_{k}^{T} G_{k}+\left(\kappa A_{k}^2+\frac{ A_{k}^{2} L \sigma_{\mathrm{ub}}}{2}\right)\left\|D_{k}\right\|\left\|G_{k}\right\|
				\leq 
				\phi(X_k) -  A_k\theta\dotp{D_k,G_k},
			\end{aligned}
			$$
			or equivalently,
			\begin{equation*}
	 A_{k}	
				\leq \frac{2(1- \theta)\dotp{D_k,G_k}}{\left({2\kappa+ L \sigma_{\mathrm{ub}}}\right)\left\|D_{k}\right\|\left\|G_{k}\right\|}.
			\end{equation*}
		By \Cref{prop:innerproduct}, this inequality is satisfied whenever 
		$$ A_k \leq \frac{2(1- \theta)\sigma_{\mathrm{lb}}}{\left({2\kappa+ L \sigma_{\mathrm{ub}}}\right)\sigma_{\mathrm{ub}}}.$$
		Hence, if $ A_{k}\leq\bar\alpha$ and the iteration is true, then \cref{eq:ass1iinoiseless} is satisfied.
	
			\item Suppose $\norm{G_k - \nabla \phi(X_k)} \leq \varepsilon_g$. Since $k < T_\varepsilon$, we have that  $\norm{\nabla{\phi(X_k)}} > \varepsilon \geq \frac{\varepsilon_g}{\eta}$ by \Cref{ass:inequality}. Therefore, $\norm{G_k - \nabla \phi(X_k)} \leq \eta\norm{\nabla\phi(X_k)}$. Combined with the fact that 
			$$ A_k \leq \bar{\alpha} \leq \frac{2\left((1-\theta)(1-\eta)\frac{\sigma_{\mathrm{lb}}}{\sigma_{\mathrm{ub}}}-\eta\right)}{L\sigma_{\mathrm{ub}}(1-\eta)},$$
			by \Cref{ass:Lipschitz} and Lemma 4.23 of \cite{berahas2021global} (applied with $\varepsilon_f = 0$), we have that \cref{eq:ass1iinoiseless} is satisfied. 
		\end{enumerate}
		Now, recalling the definitions of $E_k$ and $E_k^+$ and using the fact that $E_{k} + E_{k}^+ \leq 2\varepsilon_f$ (since $I_k = 1$), \cref{eq:ass1iinoiseless} implies
		$$
		f(X_k -  A_kD_k) \leq f(X_k) -  A_k\theta\dotp{D_k,G_k} +E_k+E_k^+ \leq f(X_k) -  A_k\theta\dotp{D_k,G_k} +2\varepsilon_f,
		$$ 
		and so the $k$th iteration is successful.

		\item[5.]  On any unsuccessful iteration, $Z_{k+1}=Z_k$, and so the claim is trivial in that case. 
		Suppose the $k$th iteration is successful. 
		Then, by the definition of a successful iteration,
		$$f(X_{k+1}) - f(X_k) \leq -  A_k\theta\dotp{D_k,G_k}+ 2\varepsilon_f \leq  2\varepsilon_f,$$ 
		since $\dotp{D_k,G_k}\geq 0$ by construction. 
		Thus, $\phi(X_{k+1}) - \phi(X_k) \leq 2\varepsilon_f + E_{k}+E_{k}^+$.
		Because $Z_{k+1} - Z_k = \phi(X_{k+1}) - \phi(X_k)$, we have proven the claim.
	\end{itemize}
 \qed
\end{proof}

\end{document}